\documentclass[11pt, notitlepage]{article}
\usepackage{amssymb,amsmath,comment}
\usepackage{amsthm}
\catcode`\@=11 \@addtoreset{equation}{section}
\def\thesection{\arabic{section}}

\def\theequation{\thesection.\arabic{equation}}
\catcode`\@=12
\usepackage{colortbl}
\usepackage{hyperref}
\usepackage[mathscr]{eucal}
\usepackage{epsf}
\usepackage{esint}
\usepackage{a4wide}
\def\R{\mathbb{R}}

\newcommand{\e}{\varepsilon}

\newcommand{\Om} {\Omega}

\newcommand{\noi} {\noindent}

\usepackage[all]{xy}
\catcode`\@=11
\def\theequation{\@arabic{\c@section}.\@arabic{\c@equation}}
\catcode`\@=12

\newtheorem{Theorem}{Theorem}[section]
\newtheorem{Lemma}[Theorem]{Lemma}
\newtheorem{prop}[Theorem]{Proposition}
\newtheorem{Corollary}[Theorem]{Corollary}
\newtheorem{Remark}[Theorem]{Remark}
\newtheorem{Definition}[Theorem]{Definition}

\begin{document}

{\vspace{0.01in}}

\title{Higher H\"older regularity for fractional $(p,q)$-Laplace equations}

\author{Prashanta Garain and Erik Lindgren}

\maketitle

\begin{abstract}\noindent
We study the fractional $(p,q)$-Laplace equation
$$
(-\Delta_p)^s u +(-\Delta_q)^t u= 0
$$ for $s,t\in(0,1)$ and $p,q\in(1,\infty)$.  We establish Hölder estimates with an explicit exponent. As a consequence, we derive a Liouville-type theorem. Our approach builds on techniques previously developed for the fractional 
$p$-Laplace equation, relying on a Moser-type iteration for difference quotients. 
\end{abstract}

\maketitle

\noi {Keywords: Fractional $p$-Laplacian, nonlocal elliptic equations of mixed order, degenerate and singular nonlocal equations, H\"older regularity, Moser iteration.}

\noi{\textit{2020 Mathematics Subject Classification: 35B65, 35J75, 35R09.}}

\bigskip

\tableofcontents

\section{Introduction}
We study the pointwise regularity of the nonlinear, nonlocal and anisotropic equation 
\begin{equation}
\label{eq:maineq}
(-\Delta_p)^s u + (-\Delta_q)^t u=0.
\end{equation}
Here $s,t\in(0,1)$, $p,q\in(1,\infty)$ and 
\begin{equation}\label{fracplap}
(-\Delta_p)^{s}u(x)=\text{P.V.}\int_{\mathbb{R}^N}\frac{|u(x)-u(y)|^{p-2}(u(x)-u(y))}{|x-y|^{N+ps}}\,dy,
\end{equation}
is the fractional $p$-Laplace operator, where P.V. denotes the principal value. This operator
appears naturally in the Euler-Lagrange equation for the functional
\[
u\mapsto \iint_{\mathbb{R}^N\times \mathbb{R}^N} \frac{|u(x)-u(y)|^p}{|x-y|^{N+s\,p}}\,dx\,dy.
\]
The main purpose of this article is to establish H\"older regularity of weak solutions of equation \eqref{eq:maineq}, with an explicit H\"older exponent. We prove that these are locally H\"older continuous for any exponent strictly less than the threshold
 $$\min\left(1,\max\left(\frac{sp}{p-1},\frac{tq}{q-1}\right)\right).
 $$The main result is presented in detail in Theorem \ref{teo:1higher} in the next subsection.

\subsection{Main result}
Our main result is presented below. For relevant notation and definitions, see Section \ref{sec:prel}. The proof of the theorem is divided into three different cases depending on if the exponents $p$ and $q$ are larger or smaller than $2$. See Sections \ref{sec:I}--\ref{sec:III}. In the statement of the theorem, we assume, without loss of generality, that $\frac{sp}{p-1}\geq \frac{tq}{q-1}$. If this doesn't hold one simply has to exchange the roles of $(s,p)$ and $(t,q)$ in order to obtain a proper statement.
\begin{Theorem}\label{teo:1higher}
Let $\Omega\subset\mathbb{R}^N$ be an open and bounded set and $s,t\in(0,1),\,A>0$. Assume $1<q\leq 2\leq p$ or $1<p\leq 2\leq q$ or $p,q\in(1,2]$ and that $\frac{sp}{p-1}\geq \frac{tq}{q-1}$. Suppose $u\in W^{s,p}_{\mathrm{loc}}(\Omega)\cap L_{sp}^{p-1}(\mathbb{R}^N)\cap  W^{t,q}_{\mathrm{loc}}(\Omega)\cap L_{tq}^{q-1}(\mathbb{R}^N)$ is a local weak solution of \eqref{eq:maineq} in $\Omega$ as in Definition \ref{subsupsolution}. 
Then $u\in C^{\Gamma-\e}_{\rm loc}(\Omega)$ for every $0<\e<\Gamma$, where $$\Gamma = \min\left(1,\frac{sp}{p-1}\right).$$ 
\par
In particular, for every $0<\e<\Gamma$ and $0<r\leq 1$ such that $B_{2r}(x_0)\Subset\Omega$, there exist constants $C=C(N,p,q,s,t,\e)>0$, $\sigma=\sigma(s,p,t,q,\e)>0$ and $\beta=\beta(N,s,p,t,q,\e)\geq 1$ such that
\begin{equation}\label{regestf}
[u]_{C^{\Gamma-\e}(B_{\sigma {r}}(x_0))}\leq \frac{C(1+r^{sp-tq}\mathcal{M}_r^{q-p})^{\beta}}{r^{\Gamma-\e}}\mathcal{M}_r,
\end{equation}
where
\begin{equation}\label{MR}
\mathcal{M}_r=\|u\|_{L^\infty(B_r(x_0))}+\mathrm{Tail}_{p-1,sp}(u;x_0,r)+\mathrm{Tail}_{q-1,tq}(u;x_0,r)+1.
\end{equation} 
\end{Theorem}

\begin{Remark}\label{mthmrmk1}
When at least one of the following conditions along with $\frac{sp}{p-1}\geq \frac{tq}{q-1}$ holds
\begin{enumerate}
\item $1<q\leq 2\leq p$,
\item $1<p\leq 2\leq q$ with $sp\geq tq$,
\item $p,q\in(1,2]$,
\end{enumerate}
then Theorem \ref{teo:1higher} applies for any $r>0$. This can be seen by inspecting the proof of the Theorem \ref{teo:1higher} in each of the separate cases.
\end{Remark}

\begin{Remark}\label{mthmrmk2}
In Theorem \ref{teo:1higher}, if we instead consider a local weak solution $u$ of 
$$
(-\Delta_p)^s u +A (-\Delta_q)^t u=0\text{ in }\Om,
$$ with $A>0$, then the estimate \eqref{regestf} will be replaced with
\begin{equation*}
[u]_{C^{\Gamma-\e}(B_{\sigma {r}}(x_0))}\leq \frac{C(1+(Ar^{sp-tq}\mathcal{M}_r^{q-p})^{\beta}+k A^{-\beta})}{r^{\Gamma-\e}}\mathcal{M}_r,
\end{equation*}
where $\mathcal{M}_r$ is as defined in \eqref{MR}. Here $k=1$ if $1<p\leq 2\leq q$ and $sp<tq$, otherwise, $k=0$. Moreover, in this case, Remark \eqref{mthmrmk1} will also hold.
\end{Remark} 
 
As a corollary of the Theorem \ref{teo:1higher} combined with Remark \ref{mthmrmk1}, we obtain the following Liouville-type result.
\begin{Corollary}\label{cor:liouville}
Let $u\in W^{s,p}(\R^N)\cap L_{sp}^{p-1}(\mathbb{R}^N)\cap  W^{t,q}(\R^N)\cap L_{tq}^{q-1}(\mathbb{R}^N)$ be a globally bounded local weak solution of \eqref{eq:maineq} in $\R^N$. Assume in addition $\beta(sp-tq)<\Gamma$ and that one of the alternatives along with $\frac{sp}{p-1}\geq \frac{tq}{q-1}$ holds:
\begin{enumerate}
\item  $1<q\leq 2\leq p$, 
\item $1<p\leq 2\leq q$ and $sp\geq tq$,
\item $p,q\in(1,2]$.
\end{enumerate}
Here $\beta$ is the $\beta$ in Theorem \ref{teo:1higher}. Then $u$ is constant in $\mathbb{R}^N$.
\end{Corollary}
\begin{proof}
Since $u$ is globally bounded, $\mathcal{M}_r$ is uniformly bounded. Taking Remark \ref{mthmrmk1} into account, we may apply the estimate \eqref{regestf} from Theorem \ref{teo:1higher} for any $r>0$. Letting $r\to \infty$ yields 
$$
[u]_{C^{\Gamma-\e}(B_{R})}=0
$$
for any $R>0$ and some $\e>0$. We conclude that $u$ is constant in $\mathbb{R}^N$.
\end{proof}
\subsection{Comments on the main result}
The results presented in this paper are not to be regarded as optimal in any sense. In fact, it is commonly anticipated that solutions to \eqref{eq:maineq} should be continuously differentiable. However, this remains an open question, even in the simpler setting involving a single fractional $p$-Laplace operator in the equation. On the other hand, for the inhomogeneous version of \eqref{eq:maineq} with a bounded inhomogeneity, the optimal regularity one might hope for is $\frac{sp}{p-1}$; see, for instance, \cite{BLS} for further details. We also wish to remark that the regularity theory for the fractional $p$-Laplace equation has undergone significant advancements in recent years. Without doubt, many of the techniques developed to refine the results in \cite{BLS} and \cite{GL} are likely to influence the study of equations like \eqref{eq:maineq} as well. See, for example, \cite{biswas2025lipschitz, BDLMBS3, BDLMBS1, BDLMBS2}.

\subsection{Known results}
In recent years, there has been a surge of interest around the fractional $p$-Laplace equation
\begin{equation}\label{fracp}
(-\Delta_p)^s u=0.
\end{equation}
The available literature is vast and we only mention a fraction of the available references. For $p=2$,  we refer to \cite{Kasscvpde} for the  local H\"older continuity for weak solutions and to \cite{Silvestre} for the viscosity setting.

In general case when $1<p<\infty$, local H\"older regularity and Harnack inequalities for weak solutions are proved in \cite{Cozzi,Kuusi1,Kuusi2}. See also \cite{Erik16}, for the viscosity setting. In \cite{Antoni4,Antoni1, Antoni2,Antoni3},  H\"older continuity up to the boundary and fine boundary regularity has been investigated.

Explicit H\"older exponents have been obtained  in \cite{BLS} and \cite{GL}. For improvements of these results, see \cite{BDLMBS1, BDLMBS3, BDLMBS2}. Another recent improvement in this direction was made in \cite{biswas2025lipschitz}, where the authors were able to reach Lipschitz regularity in the case $sp\geq (p-1)$ and $\frac{sp}{p-1}$-regularity in the case $sp<(p-1)$. Even more recently, in \cite{GJS25}, the $C^{1,\alpha}$-regularity was proved when $p\in [2,2/(1-s))$.

We also wish to mention \cite{DKLN} and \cite{DN} where sharp regularity results for equation \eqref{fracp} with a right hand side belonging to a Lorentz space are studied, both in terms of H\"older regularity and higher differentiability.

 The state of the art in terms of regularity results for equation like \eqref{eq:maineq} is not very extensive.  In \cite{GKS3}, local H\"older regularity is studied both for stationary and evolutionary problems. When $2\leq q\leq p<\infty$ and $0<t\leq s<1$, higher H\"older regularity is obtained in \cite{GKS}. In \cite{biswas2025lipschitz}, this result is improved without any specific assumptions on the parameters. In particular, when the optimal regularity exponent $\Gamma\neq 1$, their results is stronger than the one in the present paper. See also \cite{BS25}, where these results are improved in several directions. Regularity up to the boundary has been obtained in \cite{DJKTarxiv, GKScrmath}.

Equations of the type \eqref{eq:maineq} with a so-called modulating coefficient $a(x,y)$ in front of the second operator are often referred to as nonlocal double-phase equations. One of the earliest results concerning such equations is found in \cite{DFP}, where H\"older continuity is established for viscosity solutions, allowing for a very general $a$.  In the context of weak solutions, local boundedness and H\"older continuity were obtained in \cite{BOS}. See also \cite{FZ}, for H\"older regularity and the connection between weak and viscosity solutions. Higher H\"older regularity was studied in \cite{BKK}. Noteworthy is that in \cite{BKK}, higher differentiability is obtained as well. Moreover, a more modulating coefficient $a$ is allowed to appear in both of the operators. Nonlocal self-improving properties have been studied in \cite{SM}. Regarding global regularity, that is, regularity up to the boundary, we refer the reader to \cite{GKS2, GKS}.

There has also been a development of the regularity theory for nonlocal equations with general or nonstandard growth that includes \eqref{eq:maineq} as a special case. See \cite{BKO23} for the local H\"older regularity, \cite{FSV22} and \cite{FSV23} for both local and global regularity. A Harnack inequality is studied in \cite{CKW23,FZ23}. In \cite{KLL25}, a Wiener-type criterion in this setting is established.

\subsection{Plan of the paper}

In Section \ref{sec:prel}, notation, definitions and function spaces are introduced. The proof of the main theorem is contained in Sections \ref{sec:I}-\ref{sec:III}. The case $1<q\leq 2\leq p$ is treated in Section \ref{sec:I}, the case $1<p\leq 2\leq q$ is covered in Section \ref{sec:II}, and finally in Section \ref{sec:III}, we treat the case $p,q\in(1,2]$. In all three cases, different ingredients from the two papers \cite{BLS, GL} are combined.

\section{Preliminaries}\label{sec:prel}
In this section, we present the relevant notation together with some auxiliary results  required in the paper. We introduce the following notation that will be used throughout the paper.
\subsection{Notation}
\begin{itemize}
\item Throughout the paper, we work in the Euclidean space $\R^N$, with $N\geq 1$.
\item $B_r(x_0)$ denotes the ball of radius $r$  centered  at $x_0$. When $x_0=0$, we write $B_r$ in place of $B_r(0)$.

\item If a constant $C$ depends on the parameters $r_1,\ldots,r_k$, we write $C=C(r_1,\ldots,r_k)$.

\item For $1<q<\infty$, we define the monotone function $J_q:\mathbb{R}\to\mathbb{R}$ by
\begin{equation}\label{jp}
J_q(t)=|t|^{q-2}t.
\end{equation}
\item We use the notation
\begin{equation}\label{dmu}
d\mu_{l,m}=\frac{dx dy}{|x-y|^{N+lm}}\text{ for }(l,m)\in\{(p,s),(q,t)\},
\end{equation}
for the relevant kernels used in this paper.
\item For  $0< \delta\leq 1$, we denote $\delta$-H\"older seminorm by 
$$
[u]_{C^\delta(\Omega)}=\sup_{x\neq y\in \Omega}\frac{|u(x)-u(y)|}{|x-y|^{\delta}}.
$$

\item $\Gamma=\min\big(\frac{sp}{p-1},1\big)$ for $1<p<\infty$ and $0<s<1$.

\item For a function $\psi:\mathbb{R}^N\to\mathbb{R}$ and a vector $h\in\mathbb{R}^N$, we define
\[
\psi_h(x)=\psi(x+h),\qquad \delta_h \psi(x)=\psi_h(x)-\psi(x),\qquad \delta^2_h \psi(x)=\delta_h(\delta_h \psi(x))=\psi_{2\,h}(x)+\psi(x)-2\,\psi_h(x)
\]
and observe that the following product rule holds
\begin{equation}\label{Lrule1}
\delta_h(\phi\psi)=\psi_h\delta_h\phi+\phi\delta_h\psi.
\end{equation}
\end{itemize}
\subsection{Function spaces}
We introduce two Besov-type spaces that will be necessary for our arguments. To this end, let $1\le q<\infty$ and suppose $\psi\in L^q(\mathbb{R}^N)$. Then for $0<\beta\le 1$, we define
\[
[\psi]_{\mathcal{N}^{\beta,q}_\infty(\mathbb{R}^N)}:=\sup_{|h|>0} \left\|\frac{\delta_h \psi}{|h|^{\beta}}\right\|_{L^q(\mathbb{R}^N)},
\]
and 
for $0<\beta<2$, we define
\[
[\psi]_{\mathcal{B}^{\beta,q}_\infty(\mathbb{R}^N)}:=\sup_{|h|>0} \left\|\frac{\delta_h^2 \psi}{|h|^{\beta}}\right\|_{L^q(\mathbb{R}^N)}.
\]
We then introduce the Besov-type spaces
\[
\mathcal{N}^{\beta,q}_\infty(\mathbb{R}^N)=\left\{\psi\in L^q(\mathbb{R}^N)\, :\, [\psi]_{\mathcal{N}^{\beta,q}_\infty(\mathbb{R}^N)}<+\infty\right\},\qquad 0<\beta\le 1,
\]
and
\[
\mathcal{B}^{\beta,q}_\infty(\mathbb{R}^N)=\left\{\psi\in L^q(\mathbb{R}^N)\, :\, [\psi]_{\mathcal{B}^{\beta,q}_\infty(\mathbb{R}^N)}<+\infty\right\},\qquad 0<\beta<2.
\]
We also require the {\it Sobolev-Slobodecki\u{\i} space}
\[
W^{\beta,q}(\mathbb{R}^N)=\left\{\psi\in L^q(\mathbb{R}^N)\, :\, [\psi]_{W^{\beta,q}(\mathbb{R}^N)}<+\infty\right\},\qquad 0<\beta<1<q<\infty,
\]
where the seminorm $[\,\cdot\,]_{W^{\beta,q}(\mathbb{R}^N)}$ is defined by
\[
[\psi]_{W^{\beta,q}(\mathbb{R}^N)}=\left(\iint_{\mathbb{R}^N\times \mathbb{R}^N} \frac{|\psi(x)-\psi(y)|^q}{|x-y|^{N+\beta\,q}}\,dx\,dy\right)^\frac{1}{q}.
\]
We endow these spaces  with  the norms  
\[
\|\psi\|_{\mathcal{N}^{\beta,q}_\infty(\mathbb{R}^N)}=\|\psi\|_{L^q(\mathbb{R}^N)}+[\psi]_{\mathcal{N}^{\beta,q}_\infty(\mathbb{R}^N)},
\] 
\[
\|\psi\|_{\mathcal{B}^{\beta,q}_\infty(\mathbb{R}^N)}=\|\psi\|_{L^q(\mathbb{R}^N)}+[\psi]_{\mathcal{B}^{\beta,q}_\infty(\mathbb{R}^N)},
\]
and
\[
\|\psi\|_{W^{\beta,q}(\mathbb{R}^N)}=\|\psi\|_{L^q(\mathbb{R}^N)}+[\psi]_{W^{\beta,q}(\mathbb{R}^N)}.
\]
We will define the space $W^{\beta,q}(\Omega)$ for an open subset $\Omega\subset \mathbb{R}^N$,
\[
W^{\beta,q}(\Omega)=\left\{\psi\in L^q(\Omega)\, :\, [\psi]_{W^{\beta,q}(\Omega)}<+\infty\right\},\qquad 0<\beta<1,
\]
where we define
\[
 [\psi]_{W^{\beta,q}(\Omega)}=\left(\iint_{\Omega\times \Omega} \frac{|\psi(x)-\psi(y)|^q}{|x-y|^{N+\beta\,q}}\,dx\,dy\right)^\frac{1}{q}.
\]
\subsection{Embedding inequalities} The following result is \cite[Lemma 2.3]{Brolin}.
\begin{Lemma}\label{emb1}
Let $0<\beta<1$ and $1\leq q<\infty$. Then the embedding
$$
\mathcal{B}_\infty^{\beta,q}(\mathbb{R}^N)\hookrightarrow \mathcal{N}_\infty^{\beta,q}(\mathbb{R}^N) 
$$
is continuous.
More precisely, for every $\psi\in \mathcal{B}_\infty^{\beta,q}(\mathbb{R}^N)$, we have
$$
[\psi]_{\mathcal{N}_\infty^{\beta,q}(\mathbb{R}^N)}\leq\frac{C}{1-\beta}\Big([\psi]_{\mathcal{B}_\infty^{\beta,q}(\mathbb{R}^N)}+\|\psi\|_{L^q(\mathbb{R}^N)}\Big),
$$
for some constant $C=C(N,q)>0$. Moreover, for every $h_0>0$, we have
$$
\sup_{0<h<h_0}\left\|\frac{\delta_h \psi}{|h|^{\beta}}\right\|_{L^q(\mathbb{R}^N)}\leq \frac{C}{1-\beta}\Big\{ \sup_{0<h<h_0}\left\|\frac{\delta^{2}_h \psi}{|h|^{\beta}}\right\|_{L^q(\mathbb{R}^N)} +\big(h_0^{-\beta}+1\big)\|\psi\|_{L^q(\mathbb{R}^N)} \Big\}.
$$
\end{Lemma}
The following embedding result is \cite[Theorem 2.8]{BLS}.
\begin{Theorem}\label{emb2}
Let $0<\beta<1$ and $1\leq q<\infty$ be such that $\beta q>N$. If $\psi\in \mathcal{N}_\infty^{\beta,q}(\mathbb{R}^N)$, then $\psi\in C^\alpha_{\mathrm{loc}}(\mathbb{R}^N)$ for every $0<\alpha<\beta-\frac{N}{q}$. More precisely, we have
$$
\sup_{x\neq y}\frac{|\psi(x)-\psi(y)|}{|x-y|^\alpha}\leq C\Big([\psi]_{\mathcal{N}_\infty^{\beta,q}(\mathbb{R}^N)}\Big)^\frac{\alpha q+N}{\beta q}\Big(\|\psi\|_{L^q(\mathbb{R}^N)}\Big)^\frac{(\beta-\alpha)q-N}{\beta q},
$$
with $C=C(N,q,\alpha,\beta)>0$ which blows up as $\alpha\nearrow \beta-\frac{N}{q}$.
\end{Theorem}

\subsection{Tail spaces and weak solutions}
We introduce the {\it tail space} as
\[
L^{q}_{\alpha}(\mathbb{R}^N)=\left\{u\in L^{q}_{\rm loc}(\mathbb{R}^N)\, :\, \int_{\mathbb{R}^N} \frac{|u(x)|^q}{1+|x|^{N+\alpha}}\,dx<+\infty\right\},\qquad q>0 \mbox{ and } \alpha>0,
\]
and define for every $x_0\in\mathbb{R}^N$, $R>0,\,\beta>0$ and $u\in L^q_{\alpha}(\mathbb{R}^N)$ by
\begin{equation*}
\mathrm{Tail}_{q,\alpha}(u;x_0,R)=\left[R^\alpha\,\int_{\mathbb{R}^N\setminus B_R(x_0)} \frac{|u(x)|^q}{|x-x_0|^{N+\alpha}}\,dx\right]^\frac{1}{q}.
\end{equation*}
We observe that the quantity above is always finite, for a function $u\in L^q_{\alpha}(\mathbb{R}^N)$.
We now introduce the solution space. For $p,q\in (1,\infty)$, $s,t\in(0,1)$, and $\Omega$ a bounded and open subset of $\R^N$, we define the solution space $X(\Omega)$ as
    $$
    X(\Omega)=W^{s,p}_{\mathrm{loc}}(\Omega)\cap L_{sp}^{p-1}(\mathbb{R}^N)\cap  W^{t,q}_{\mathrm{loc}}(\Omega)\cap L_{tq}^{q-1}(\mathbb{R}^N).
        $$
\begin{Definition}\label{subsupsolution}(Local weak solution)
Let $p,q\in(1,\infty)$, $A>0$ and $s,t\in(0,1)$. Suppose $\Omega\subset\mathbb{R}^N$ is an open and bounded set. We say that $u\in X(\Omega)$ is a local weak solution of the equation 
\begin{equation}\label{eq:maineqnew}
(-\Delta_p)^s u+A(-\Delta_q)^t u=0\text{ in } \Omega,
\end{equation}
if 
\begin{equation}\label{wksol}
\begin{gathered}
\sum_{(l,m,k)\in\{(p,s,1),(q,t,A)\}}k\int_{\mathbb{R}^N}\int_{\mathbb{R}^N}J_l(u(x)-u(y))(\varphi(x)-\varphi(y))\,d\mu_{l,m}=0,
\end{gathered}
\end{equation}
for $\varphi\in W^{s,p}(\Omega)\cap W^{t,q}(\Omega)$  with compact support in $\Omega$.   Recall that  $J_l(t)=|t|^{l-2}t$ and $d\mu_{l,m}=\frac{dx dy}{|x-y|^{N+lm}}$,  as defined in \eqref{jp} and \eqref{dmu} respectively.
\end{Definition}

\section{Case I: $1<q\leq 2\leq p$}\label{sec:I}
Throughout this section, it will be assumed that $1<q\leq 2\leq p$ and $s,t\in(0,1)$ satisfy $\frac{sp}{p-1}\geq\frac{tq}{q-1}$ unless stated otherwise.
\subsection{Improved Besov regularity and almost $C^s$-regularity}
 The following improved Besov-type regularity serves as a starting point. 
\begin{prop}
\label{prop:improve}
Let $A>0$ and assume that $u\in X(B_2)$ is a local weak solution of \eqref{eq:maineqnew} in $B_2$  satisfying
\begin{equation}
\label{bounds}
\|u\|_{L^\infty(B_1)}\leq 1, \qquad 
\int_{\mathbb{R}^N\setminus B_1} \frac{|u(y)|^{l-1}}{|y|^{N+lm}}\,dy\leq 1,
\end{equation}
where $(l,m)\in\{(p,s),(q,t)\}$. Then for any $p\leq\mu <\infty$ and $0<h_0<\frac{1}{10}$ and $R$ such that $4\,h_0<R\le 1-5\,h_0$,  
we have
\begin{equation}\label{pro1}
\begin{split}
\sup_{0<|h|< h_0}\left\|\frac{\delta^2_h u}{|h|^{s}}\right\|_{L^{\mu+1}(B_{R-4\,h_0})}^{{\mu}+1}\leq C(A+1)\,\left(\sup_{0<|h|< h_0}\left\|\frac{\delta^2_h u }{|h|^s}\right\|_{L^{\mu}(B_{R+4h_0})}^{\mu}+1\right),
\end{split}
\end{equation}
for some positive constant $C=C(N,s,p,t,q,\mu,h_0)$.
\end{prop}
\begin{proof}
 Define 
$
r=R-4h_0$ and take  $\varphi\in W^{s,p}(B_R)\cap W^{t,q}(B_R)$ vanishing outside $B_{\frac{R+r}{2}}$ and  $0<|h|<h_0$. Testing \eqref{wksol} with $\varphi$ and $\varphi_{-h}$ and performing a change of variable yields
\begin{equation}
\label{differentiated1}
\frac1h \sum_{(l,m,k)\in\{(p,s,1),(q,t,A)\}}\int_{\mathbb{R}^N} \int_{\mathbb{R}^N} k\Big(J_l(u_h(x)-u_h(y))-J_l(u(x)-u(y))\Big)\,\Big(\varphi(x)-\varphi(y)\Big)\,d\mu_{l,m}=0.
\end{equation}
  We insert
 \[
\varphi=J_{\beta+1}\left(\frac{\delta_h u}{|h|^{\theta}}\right)\,\eta^p, 
\]
in \eqref{differentiated1}, for some $\beta\geq 2$ and some $\theta$ such that $1+\theta\beta\geq 0$, where 
$\eta\in C^\infty_0(B_R)$ is such that
\begin{equation}\label{eta1p}
0\le \eta\le 1\text{ in }B_R,\qquad\eta\equiv 1\text{ in }B_r,\qquad \eta\equiv 0 \mbox{ in } \mathbb{R}^N\setminus B_{\frac{R+r}{2}},\qquad |\nabla \eta|\le \frac{C}{R-r}=\frac{C}{4\,h_0},
\end{equation}
for some constant $C=C(N)>0.$ This gives
\begin{equation}\label{deqn11}
\begin{split}
&\sum_{(l,m,k)\in\{(p,s,1),(q,t,A)\}}\iint_{\mathbb{R}^N\times\mathbb{R}^N} k\frac{\Big(J_l(u_h(x)-u_h(y))-J_l(u(x)-u(y))\Big)}{|h|^{1+\theta\,\beta}}\\
&\times\Big(J_{\beta+1}(u_h(x)-u(x))\,\eta(x)^p-J_{\beta+1}(u_h(y)-u(y))\,\eta(y)^p\Big)\,d\mu_{l,m}=0.
\end{split}
\end{equation}
This can be written as
\begin{equation}
\label{Il1}
\sum_{(l,m,k)\in\{(p,s,1),(q,t,A)\}}(\mathcal{I}_{1,l,m,k}+\mathcal{I}_{2,l,m,k}+\mathcal{I}_{3,l,m,k})=0,
\end{equation}
where
\[
\begin{split}
\mathcal{I}_{1,l,m,k}:=\iint_{B_R\times B_R}& k\frac{\Big(J_l(u_h(x)-u_h(y))-J_l(u(x)-u(y))\Big)}{|h|^{1+\theta\,\beta}}\\
&\times\Big(J_{\beta+1}(u_h(x)-u(x))\,\eta(x)^p-J_{\beta+1}(u_h(y)-u(y))\,\eta(y)^p\Big)d\mu_{l,m},
\end{split}
\]
\[
\begin{split}
\mathcal{I}_{2,l,m,k}:=\iint_{B_\frac{R+r}{2}\times (\mathbb{R}^N\setminus B_R)}& k\frac{\Big(J_l(u_h(x)-u_h(y))-J_l(u(x)-u(y))\Big)}{|h|^{1+\theta\,\beta}}\\
&\times J_{\beta+1}(u_h(x)-u(x))\,\eta(x)^p\,d\mu_{l,m},
\end{split}
\]
and
\[
\begin{split}
\mathcal{I}_{3,l,m,k}:=-\iint_{(\mathbb{R}^N\setminus B_R)\times B_\frac{R+r}{2}}& k\frac{\Big(J_l(u_h(x)-u_h(y))-J_l(u(x)-u(y))\Big)}{|h|^{1+\theta\,\beta}}\\
&\times J_{\beta+1}(u_h(y)-u(y))\,\eta(y)^p\,d\mu_{l,m},
\end{split}
\]
where we used that $\eta$ vanishes identically outside $B_{\frac{R+r}{2}}$.

We estimate $\mathcal{I}_{j,l,m,k}$ for $j=1,2,3$ and $(l,m,k)\in\{(p,s,1),(q,t,A)\}$, starting with the case $(l,m,k)=(q,t,A)$. To this end, we remark that $\eta^p=(\eta^\frac{p}{2})^2$ and $\eta^\frac{p}{2}$ is Lipschitz and enjoys the same properties as $\eta$ has, including \eqref{eta1p}. Taking into account that $1<q\leq 2$ and that $\|u\|_{L^\infty(B_1)}\leq 1$ we may proceed as in the proof of the estimate of $\mathcal{I}_1$ in \cite[p. 5767]{GL} and obtain\footnote{Note that in \cite{GL}, only the case $q<2$ is considered. However, when $q=2$ the equation is linear and the estimate is standard.}
\begin{equation}\label{estI1q1}
\begin{split}
\mathcal{I}_{1,q,t,A}
&\geq {c}A\left[\frac{|\delta_h u|^\frac{\beta-1}{2}\,\delta_h u}{|h|^\frac{1+\theta\,\beta}{2}}\,\eta^\frac{p}{2}\right]^2_{W^{\frac{tq}{2},2}(B_R)}-CA\int_{B_R} \frac{|\delta_h u(x)|^{q+\beta-1}}{|h|^{1+\theta\,\beta}} dx\\
& -CA\,\int_{B_R}\, \frac{|\delta_h u(x)|^{\beta+1}}{|h|^{1+\theta\,\beta}}dx.
\end{split}
\end{equation}
Here $c=c({p},q,\beta)>0$ and $C=C(N,p,t,q,\beta,h_0)>0$. Further, taking into account \eqref{bounds} and proceeding along the lines of the proof of the estimates (3.17) and (3.18) from \cite[page 5768]{GL}, we get
\begin{equation}
\label{estI2q1}
|\mathcal{I}_{2,q,t,A}|\leq CA\int_{B_R}\frac{|\delta_h u(x)|^{\beta}}{|h|^{1+\theta\,\beta}} dx,
\end{equation}
and
\begin{equation}
\label{estI3q1}
|\mathcal{I}_{3,q,t,A}|\leq CA\int_{B_R}\frac{|\delta_h u(x)|^{\beta}}{|h|^{1+\theta\,\beta}} dx,
\end{equation}
where $C=C(N,t,q,h_0)>0$ is a constant. Next, we estimate $\mathcal{I}_{j,l,m,k}$ for $(l,m,k)=(p,s,1)$.

Again, since $p\geq 2$, taking into account the properties of $\eta$ and proceeding along the lines of the proof of the estimates of $\mathcal{I}_1, \mathcal{I}_{11}$ and $\mathcal{I}_{12}$ from \cite[pages 814-819]{BLS}, we get
\begin{equation}\label{estI1p1}
\begin{split}
\mathcal{I}_{1,p,s,1}\ge c& \left[\frac{|\delta_h u|^\frac{\beta-1}{p}\,\delta_h u}{|h|^\frac{1+\theta\,\beta}{p}}\,{\eta}\right]^p_{W^{s,p}(B_R)}-C\,\mathcal{I}_{11,p,s,1}-C\,\mathcal{I}_{12,p,s,1}
\end{split}
\end{equation}
where $c=c(p,\beta)>0$ and $C=C(p,\beta)>0$. The terms $\mathcal{I}_{11,p,s,1}$ and $\mathcal{I}_{12,p,s,1}$ are defined by
\begin{equation}
\label{I111}
\begin{split}
\mathcal{I}_{11,p,s,1}&\quad:=\,\iint_{B_R\times B_R} \left(|u_h(x)-u_h(y)|^\frac{p-2}{2}+|u(x)-u(y)|^\frac{p-2}{2}\right)^2\\
&\times\left|\eta(x)^\frac{p}{2}-\eta(y)^\frac{p}{2}\right|^2\, \frac{|\delta_h u(x)|^{\beta+1}+|\delta_h u(y)|^{\beta+1}}{|h|^{1+\theta\,\beta}}\,d\mu_{p,s},
\end{split}
\end{equation}
and
\begin{equation}
\label{eq:I2original}
\begin{split}
\mathcal{I}_{12,p,s,1}&:=\,\iint_{B_R\times B_R}\, \left(\frac{|\delta_h u(x)|^{\beta-1+p}}{|h|^{1+\theta\,\beta}}+\frac{|\delta_h u(y)|^{\beta-1+p}}{|h|^{1+\theta\,\beta}}\right)\, |\eta(x)-\eta(y)|^p\,d\mu_{p,s}
\end{split}
\end{equation}
and enjoy the estimates
\begin{equation}\label{estI111}
|\mathcal{I}_{11,p,s,1}|\leq C\,\left(\int_{B_R}\frac{|\delta_h u(x)|^{\frac{\beta\,\mu}{\mu-p+2}}}{|h|^{(1+\theta\,\beta)\frac{\mu}{\mu-p+2}}}\, dx +\sup_{0<|h|< h_0}\left\|\frac{\delta^2_h u}{|h|^s}\right\|_{L^{\mu}(B_{R+4h_0})}^{\mu}+1\right)
\end{equation}
and
\begin{equation}\label{estI121}
|\mathcal{I}_{12,p,s,1}|\leq C\left( \int_{B_R} \frac{|\delta_h u|^{\frac{\beta \mu}{\mu-p+2}}}{|h|^{(1+\theta\,\beta)\,\frac{\mu}{\mu-p+2}}}\,dx+1\right),
\end{equation}
where $C=C(N,s,p,\mu,h_0)>0$. Now exactly as in the estimate (4.9) in \cite[page 820]{BLS}, we obtain
\begin{equation}\label{estI23p1}
\begin{split}
|\mathcal{I}_{2,p,s,1}|+|\mathcal{I}_{3,p,s,1}|&\leq C\,\int_{B_{\frac{R+r}{2}}}\frac{|\delta_h u|^{\beta}}{|h|^{1+\theta\beta}} dx\leq C\,\left(1+\int_{B_{R}}\left|\frac{\delta_h u}{|h|^{\frac{1+\theta\beta}{\beta}}}\right|^{\frac{\beta\,\mu}{\mu-p+2}}\,dx\right)
\end{split}
\end{equation}
where $C=C(N,s,p,h_0)>0$. By inserting the estimates \eqref{estI1q1}, \eqref{estI2q1}, \eqref{estI3q1}, \eqref{estI1p1}, \eqref{estI111}, \eqref{estI121} and \eqref{estI23p1} in \eqref{Il1}, we obtain
\begin{equation}\label{estIl1}
\begin{split}
&A\left[\frac{|\delta_h u|^\frac{\beta-1}{2}\,\delta_h u}{|h|^\frac{1+\theta\,\beta}{2}}\,{\eta^\frac{p}{2}}\right]^2_{W^{\frac{tq}{2},2}(B_R)}
+
\left[\frac{|\delta_h u|^\frac{\beta-1}{p}\,\delta_h u}{|h|^\frac{1+\theta\,\beta}{p}}\,{\eta}\right]^p_{W^{s,p}(B_R)}\\
&
\leq  CA\,\Big(\int_{B_R}  \frac{|\delta_h u(x)|^{q+\beta-1}}{|h|^{1+\theta\,\beta}} +\frac{|\delta_h u(x)|^{\beta+1 }}{|h|^{1+\theta\,\beta}}+\frac{|\delta_h u(x)|^{\beta}}{|h|^{1+\theta\,\beta}}dx\Big)\\
&\qquad\qquad+C\,\left(\int_{B_{R}}\left|\frac{\delta_h u}{|h|^{\frac{1+\theta\beta}{\beta}}}\right|^{\frac{\beta\,\mu}{\mu-p+2}}\,dx+\sup_{0<|h|< h_0}\left\|\frac{\delta^2_h u}{|h|^s}\right\|_{L^{\mu}(B_{R+4h_0})}^{\mu}+1\right),
\end{split}
\end{equation}
where $C=C(N,s,p,t,q,\beta,\mu,h_0)>0$. Now, we ignore the first term in the L.H.S. of \eqref{estIl1} and estimate the second term in the L.H.S. of \eqref{estIl1} exactly as in the estimate of (4.15) in \cite[page 822]{BLS} and insert the resulting estimate in \eqref{estIl1} to obtain
\begin{equation}\label{estIlfinal}
\begin{split}
&\sup_{0<|h|< h_0}\int_{B_r}\left|\frac{\delta^2_h u}{|h|^\frac{1+s\,p+\theta\,\beta}{\beta-1+p}}\right|^{\beta-1+p}\,dx\\
&
\leq  CA\sup_{0<|h|< h_0}\,\Big(\int_{B_R}  \frac{|\delta_h u(x)|^{q+\beta-1}}{|h|^{1+\theta\,\beta}} +\frac{|\delta_h u(x)|^{\beta+1 }}{|h|^{1+\theta\,\beta}}+\frac{|\delta_h u(x)|^{\beta}}{|h|^{1+\theta\,\beta}}dx\Big)\\
&\qquad\qquad+C\sup_{0<|h|< h_0}\,\left(\int_{B_{R}}\left|\frac{\delta_h u}{|h|^{\frac{1+\theta\beta}{\beta}}}\right|^{\frac{\beta\,\mu}{\mu-p+2}}\,dx+\left\|\frac{\delta^2_h u}{|h|^s}\right\|_{L^{\mu}(B_{R+4h_0})}^{\mu}+1\right)\\
&\leq C(A+1)\,\sup_{0<|h|< h_0}\,\left(\int_{B_{R}}\left|\frac{\delta_h u}{|h|^{\frac{1+\theta\beta}{\beta}}}\right|^{\frac{\beta\,\mu}{\mu-p+2}}\,dx+\left\|\frac{\delta^2_h u}{|h|^s}\right\|_{L^{\mu}(B_{R+4h_0})}^{\mu}+1\right),
\end{split}
\end{equation}
for $C=C(N,s,p,t,q,\beta,\mu,h_0)>0$, where the last inequality above follows by using that $|u|\leq 1$ in $B_1$ and $p\geq 2$.

 Now, choosing $\beta=\mu-p+2$ and $\theta=\frac{(\mu-p+2)\,s-1}{\mu-p+2}$ and further arguing exactly same as in Step 5 of \cite[pages 822-823]{BLS}, we obtain
$$
\sup_{0<|h|< h_0}\left\|\frac{\delta^2_h u}{|h|^{s}}\right\|_{L^{\mu+1}(B_{R-4\,h_0})}^{\mu+1}\leq C(A+1)\,\left(\sup_{0<|h|< h_0}\left\|\frac{\delta^2_h u }{|h|^s}\right\|_{L^{\mu}(B_{R+4h_0})}^{\mu}+1\right),
$$
for $C=C(N,s,p,t,q,\mu,h_0)>0$.
\end{proof}
Following the proof of \cite[Theorem 4.2]{BLS}, we iterate the previous result in Proposition \ref{prop:improve} to prove the following theorem.
\begin{Theorem}
\label{teo:1}(Almost $C^s$-regularity)
Let $\Omega\subset\mathbb{R}^N$ be a bounded and open set, $A>0$. Suppose $u\in X(\Omega)$ is a local weak solution of \eqref{eq:maineqnew} in $\Omega$.  
Then $u\in C^\delta_{\rm loc}(\Omega)$ for every $0<\delta<s$. 
\par
 In particular,  for every $0<\delta<s$ and every $r>0$  such that  $B_{2r}(x_0)\Subset\Omega$, there  are  constants {$d=d(N,s,p,\delta)\geq 1$ and $C=C(N,s,p,t,q,\delta)>0$}  such that  
\begin{equation}
\label{apriori}
[u]_{C^\delta(B_{r/2}(x_0))} \frac{C(1+Ar^{sp-tq}\mathcal{M}_r^{q-p})^d}{r^\delta}\mathcal{M}_r,
\end{equation}
where 
$$
\mathcal{M}_r=\|u\|_{L^\infty(B_{r}(x_0))}+\mathrm{Tail}_{p-1,s\,p}(u;x_0,r)+\mathrm{Tail}_{q-1,t\,q}(u;x_0,r)+1.
$$
\end{Theorem}
\begin{proof}
  By \cite[Proposition 3.3]{GKS}, $u\in L^\infty_{\rm loc}(\Omega)$. We assume without loss of generality that $x_0=0$  and set 
\[
\mathcal{M}_r=\|u\|_{L^\infty(B_{r})}+\mathrm{Tail}_{p-1,s\,p}(u;0,r)+\mathrm{Tail}_{q-1,t\,q}(u;0,r)+1>0.
\]
 It is sufficient to prove that the function 
\[
u_r(x):=\frac{1}{\mathcal{M}_r}\,u(r\,x),\qquad \mbox{ for }x\in B_2,
\]
satisfies the estimate
\[
[u_r]_{C^{\delta}(B_{\frac{1}{2}})}\leq C(1+Ar^{sp-tq}\mathcal{M}_r^{q-p})^d,
\]
{for some positive constants $C=C(N,s,p,t,q,\delta)$ and $d=d(N,s,p,t,q,\delta)$.}
 Indeed, by scaling,  we would  recover  the desired estimate \eqref{apriori}. Observe that by definition, the function $u_r$ 
is a local weak solution of $(-\Delta_p)^s u+Ar^{sp-tq}\mathcal{M}_r^{q-p}(-\Delta_q)^t u=0$ in $B_2$ and satisfies
\begin{equation}
\label{assumptionur}
\|u_r\|_{L^\infty(B_1)}\leq 1,\qquad \int_{\mathbb{R}^N\setminus B_1}\frac{|u_r(y)|^{l-1}}{|y|^{N+ml}}\,  dy\leq 1,
\end{equation}
for $(l,m)\in\{(p,s),(q,t)\}$. 
In what follows, we omit the subscript $r$ and simply write $u$ in place of $u_r$. 

Fix $0<\delta<s$ and choose $i_\infty\in\mathbb{N}\setminus\{0\}$ such that
\[
s-\delta> \frac{N}{p+i_\infty}
\]
 and  define the sequence of exponents  
\[
q_i=p+i,\qquad i=0,\dots,i_\infty.
\]
 In addition, we define 
$$
h_0=\frac{1}{64\,i_\infty},\qquad R_i=\frac{7}{8}-4\,(2i+1)\,h_0=\frac{7}{8}-\frac{2i+1}{16\,i_\infty},\qquad \mbox{ for } i=0,\dots,i_\infty,
$$
 and remark that this implies 
\[
R_0+4\,h_0=\frac{7}{8}\qquad \mbox{ and }\qquad R_{i_\infty-1}-4\,h_0=\frac{3}{4}.
\] 
By applying Proposition \ref{prop:improve} with 
\[
R=R_i\qquad \mbox{ and }\qquad \mu=q_i=p+i,\qquad \mbox{ for } i=0,\ldots,i_\infty-1,
\] 
and observing that $R_i-4\,h_0=R_{i+1}+4\,h_0$, $4h_0<R_i\leq 1-5h_0$ and $q\leq p$,
we obtain the  following scheme of  inequalities
\[
\left\{\begin{array}{rcll}
\sup\limits_{0<|h|< h_0}\left\|\dfrac{\delta^2_h u}{|h|^{s}}\right\|_{L^{q_1}(B_{R_1+4h_0})}\leq& C (Ar^{sp-tq}\mathcal{M}_r^{q-p}+1)\,\sup\limits_{0<|h|< h_0}\left(\left\|\dfrac{\delta^2_h u }{|h|^s}\right\|_{L^p(B_{\frac{7}{8}})}+1\right)
\\
\sup\limits_{0<|h|<h_0}\left\|\dfrac{\delta^2_h u}{|h|^{s}}\right\|_{L^{q_i+1}(B_{R_{i+1}+4h_0})}\leq&C(Ar^{sp-tq}\mathcal{M}_r^{q-p}+1)\,\sup\limits_{0<|h|< h_0}\left(\left\|\dfrac{\delta^2_h u }{|h|^s}\right\|_{L^{q_i}(B_{R_i+4h_0})}+1\right)
\end{array}
\right.
\]
and finally
\begin{align*}
\sup_{0<|h|< {h_0}}\left\|\frac{\delta^2_h u}{|h|^{s}}\right\|_{L^{q_{i_\infty}}(B_{\frac{3}{4}})}&=\sup_{0<|h|< h_0}\left\|\frac{\delta^2_h u}{|h|^{s}}\right\|_{L^{p+i_\infty}(B_{R_{i_\infty-1}-4h_0})}\\
&\leq C(Ar^{sp-tq}\mathcal{M}_r^{q-p}+1)\sup_{0<|h|< h_0}\left(\left\|\frac{\delta^2_h u }{|h|^s}\right\|_{L^{p+i_\infty-1}(B_{R_{i_\infty-1}+4h_0})}+1\right).
\end{align*}
Here $C=C(N,s,p,t,q,\delta,h_0)>0$. Using \eqref{assumptionur} and proceeding exactly as in the proof of \cite[Lemma 3.2]{GKS}, we get 
\begin{equation}\label{fsn}
[u_r]_{W^{s,p}(B_{\frac{7}{8}+2h_0})}\leq C(Ar^{sp-tq}\mathcal{M}_r^{q-p}+1),
\end{equation}
for some $C=C(N,s,p,t,q)>0$.   By {\cite[Proposition 2.6, estimate (2.11)]{Brolin}} combined  with the relation 
$$
\delta_h u =\frac12\left(\delta_{2h} u-\delta^2_{h}u\right),
$$ 
we also have
\begin{align}\label{eq:1sttofrac}
\sup_{0<|h|< h_0}\left\|\frac{\delta^2_h u }{|h|^s}\right\|_{L^{p}(B_{\frac{7}{8}})}&\leq C\,\sup_{0<|h|<2\,h_0}\left\|\frac{\delta_h u }{|h|^s}\right\|_{L^{p}(B_{\frac{7}{8}})}\nonumber \\
&\leq C\left([u]_{W^{s,p}(B_{\frac{7}{8}+2\,h_0})}+\|u\|_{L^\infty(B_{\frac{7}{8}+2\,h_0})}\right)\\
&\leq C(Ar^{sp-tq}\mathcal{M}_r^{q-p}+1),\nonumber 
\end{align}
where $C=C(N,s,p,t,q,\delta,h_0)>0$ and where we used \eqref{assumptionur} and \eqref{fsn}.  Therefore, the  above scheme of inequalities implies 
\begin{equation}
\sup_{0<|h|< {h_0}}\left\|\frac{\delta^2_h u}{|h|^{s}}\right\|_{L^{q_{i_\infty}}(B_{\frac{3}{4}})}\leq C(Ar^{sp-tq}\mathcal{M}_r^{q-p}+1)^{i_\infty+1},
\label{otherest}
\end{equation}
where $C(N,s,p,t,q,\delta,h_0)>0$.
Take $\chi\in C_0^\infty(B_{\frac{5}{8}})$ such that 
$$
0\leq \chi\leq 1, \qquad \chi \equiv 1 \text{ in $B_{\frac{1}{2}}$},\qquad |\nabla \chi|\leq C,\qquad |D^2 \chi|\leq C,
$$
for some constant $C=C(N)>0$.  Then  we have for all $|h| > 0$ that
$$
\frac{|\delta_h\chi|}{|h|^s}\leq C,\qquad \frac{|\delta^2_h\chi|}{|h|^s}\leq C,
$$
for some constant $C=C(N)>0$.
We also note that
$$
\delta^2_h (u\,\chi)=\chi_{2h}\,\delta^2_h u+2\,\delta_h u\, \delta_h \chi_h+u\,\delta^2_h\chi.
$$
Hence, we have
\begin{equation}\label{Neq-c1}
\begin{split}
[u\,\chi]_{\mathcal{B}^{s,q_{i_\infty}}_\infty(\mathbb{R}^N)}&\leq C\left(\sup_{0<|h|< h_0}\left\|\frac{\delta^2_h (u\,\chi)}{|h|^s}\right\|_{L^{q_{i_\infty}}(\mathbb{R}^N)}+1\right)\\&\leq C\sup_{0<|h|< h_0}\,\left(\left\|\frac{\chi_{2h}\,\delta^2_h u}{|h|^s}\right\|_{L^{q_{i_\infty}}(\mathbb{R}^N)}+\left\|\frac{\delta_h u\,\delta_h\chi}{|h|^s}\right\|_{L^{q_{i_\infty}}(\mathbb{R}^N)}+\left\|\frac{u\,\delta^2_h\chi}{|h|^s}\right\|_{L^{q_{i_\infty}}(\mathbb{R}^N)}+1\right)\\
&\leq C\sup_{0<|h|< h_0}\,\left(\left\|\frac{\delta^2_h u}{|h|^s}\right\|_{L^{q_{i_\infty}}(B_{\frac{5}{8}+2\,h_0})}+\|\delta_h u\|_{L^{q_{i_\infty}}(B_{\frac{5}{8}+h_0})}+\|u\|_{L^{q_{i_\infty}}(B_{\frac{5}{8}+2h_0})}+1\right) \\
&\leq C\sup_{0<|h|< h_0}\,\left(\left\|\frac{\delta^2_h u}{|h|^s}\right\|_{L^{q_{i_\infty}}(B_{\frac{3}{4}})}+\|u\|_{L^{q_{i_\infty}}(B_{\frac{3}{4}})}+1\right)\\
&\leq C(Ar^{sp-tq}\mathcal{M}_r^{q-p}+1)^{i_\infty+1},
\end{split}
\end{equation}
by \eqref{otherest}, {where $C=C(N,s,p,t,q,\delta,h_0)>0$}. By Lemma \ref{emb1}, we have
\begin{equation}\label{Neq2-c1}
[u\,\chi]_{\mathcal{N}_\infty^{s,q_{i_\infty}}(\mathbb{R}^N)}\leq C(\,[u\,\chi]_{\mathcal{B}_\infty^{s,q_{i_\infty}}(\mathbb{R}^N)}+1) \leq C(Ar^{sp-tq}\mathcal{M}_r^{q-p}+1)^{i_\infty+1},
\end{equation}
where $C=C(N,s,p,t,q,\delta,h_0)>0$. Finally,  we note that the choice of $i_\infty$ implies 
\[
s\,q_{i_\infty}>N\qquad \mbox{and }\qquad \delta<s-\frac{N}{q_{i_\infty}}.
\] 
Therefore,  using  Theorem~\ref{emb2} with  $\beta=s$, $\alpha=\delta$ and $q=q_{i_\infty}$  we obtain
\begin{align*}
[u]_{C^\delta(B_{\frac{1}{2}})}&= [u\,\chi]_{C^\delta(B_{\frac{1}{2}})}\\
&\leq C\left([u\,\chi]_{\mathcal{N}_\infty^{s,q_{i_\infty}}(\mathbb{R}^N)}\right)^{\frac{\delta\,q_{i_\infty}+N}{s\,q_{i_\infty}}}\,\left(\|u\,\chi\|_{L^q(\mathbb{R}^N)}\right)^\frac{(s-\delta)\,q_{i_\infty}-N}{s\,q_{i_\infty}}\\
&\leq C(N,s,p,t,q,\delta,h_0)(Ar^{sp-tq}\mathcal{M}_r^{q-p}+1)^\frac{(\delta q_{i_\infty}+N)(i_\infty+1)}{sq_{i_\infty}}.
\end{align*}
Defining $d=\frac{(\delta q_{i_\infty}+N)(i_\infty+1)}{sq_{i_\infty}}+1$, concludes the proof, upon recalling that $i_\infty$ and $h_\infty$ are chosen as functions of $N,p,s$ and $\delta$.
\end{proof}
\begin{Remark}
\label{rem:flexibility}  We may use a covering argument combined with \eqref{apriori} to obtain a more general estimate under the assumptions of Theorem \ref{teo:1}:  for every $0<\sigma<1$, we have
\[
\begin{split}
[u]_{C^\delta(B_{\sigma r}(x_0))}&\leq \ \frac{C(1+Ar^{sp-tq}\mathcal{M}_r^{q-p})^d}{r^\delta}\\&\times \left(\|u\|_{L^\infty(B_{r}(x_0))}+\mathrm{Tail}_{p-1,s\,p}(u;x_0,r)+\mathrm{Tail}_{q-1,t\,q}(u;x_0,r)+1\right),
\end{split}
\]
where $C$ will not now depend on $\sigma$ as well.  Indeed, if {$0<\sigma\le \frac{1}{2}$}, then this follows directly from \eqref{apriori}  . If $\frac{1}{2}<\sigma<1$, we can cover $\overline{B_{\sigma\,r}(x_0)}$ with a finite number of balls $B_{\frac{R}{2}}(x_1),\dots,B_{\frac{R}{2}}(x_k)$, where
\[
x_i\in B_{\frac{r}{2}}(x_0)\qquad \mbox{ and }\qquad R=\frac{r}{2}.
\]
By  applying  \eqref{apriori} on each ball $B_{2R}(x_i)\Subset B_{2r}(x_0)\Subset\Omega$, we obtain  
\[\begin{split}
&[u]_{C^\delta(B_{\frac{R}{2}}(x_i))}\leq  \frac{C(1+AR^{sp-tq}\mathcal{M}_R(x_i)^{q-p})^d}{R^\delta}\\
&\times \left(\|u\|_{L^\infty(B_{R}(x_i))}+\mathrm{Tail}_{p-1,s\,p}(u;x_i,R)+\mathrm{Tail}_{q-1,t\,q}(u;x_i,R)+1\right),
\end{split}
\]
where
$$
\mathcal{M}_R(x_i)=\|u\|_{L^\infty(B_{R}(x_i))}+\mathrm{Tail}_{p-1,s\,p}(u;x_i,R)+\mathrm{Tail}_{q-1,t\,q}(u;x_i,R)+1.
$$
 Since $B_{R}(x_i)\subset B_r(x_0)$, we may sum over $i$ and use \cite[Lemma 2.3]{BLS} for the tail terms, in order to reach the desired conclusion. 
\end{Remark}
\subsection{Improved H\"older regularity}
\label{sec:higher}
{By using Theorem \ref{teo:1},} i.e., by using that a solution to the homogeneous equation is locally $\delta$-H\"older continuous for any $0<\delta<s$, we establish the following  enhanced version  of Proposition \ref{prop:improve}, where both the integrability and the differentiability are improved.
\begin{prop}\label{prop:improve2I}(Improved H\"older regularity)
Let $A>0$ and $u\in  X(B_2)$ be a local weak solution of \eqref{eq:maineqnew} in $B_2$  such that 
\[
\|u\|_{L^\infty(B_1)}\leq 1\qquad \mbox{ and }\qquad  \int_{\mathbb{R}^N\setminus B_1} \frac{|u(y)|^{l-1}}{|y|^{N+m\,l}}\,  dy\leq 1,
\]
for $(l,m)\in\{(p,s),(q,t)\}$.
Then for $0<h_0<\frac{1}{10}$, $0<\theta<1$ and $\beta>1$ such that $\frac{1+\theta\,\beta}{\beta}<1$ and for $R$ such that  $4\,h_0<R\le 1-5\,h_0$,
we have
$$
\sup_{0<|h|< h_0}\left\|\frac{\delta^2_h u}{|h|^{\frac{1+s\,p+\theta\,\beta}{\beta-1+p}}}\right\|_{L^{\beta-1+p}(B_{R-4\,h_0})}^{\beta-1+p}\leq C(1+A)^d\,\sup_{0<|h|< h_0}\left(\left\|\frac{\delta^2_h u }{|h|^\frac{1+\theta\, \beta}{\beta}}\right\|_{L^\beta(B_{R+4\,h_0})}^\beta+1\right).
$$
Here $d=d(N,s,p,h_0)\geq 1$ and $C=C(N,s,p,t,q,\beta,h_0)>0$.
\end{prop}
\begin{proof}  The proof is a slight variant of the proof of of Proposition \ref{prop:improve}.  The only difference is that we estimate the term $\mathcal{I}_{11,p,s,1}$ defined in \eqref{I111} differently. We fix $\varepsilon=\varepsilon(p,s)$ such that 
\[
0<\varepsilon<\min\left\{s,\frac{2(1-s)}{p-2}\right\}.
\]
Using Theorem~\ref{teo:1} and Remark \ref{rem:flexibility}, we have
\begin{equation}\label{csuse1}
\begin{split}
[u]_{C^{s-\varepsilon}(B_R)}&\leq C(AR^{sp-tq}\mathcal{M}_R^{q-p}+1)^d\leq C(1+A)^d,
\end{split}
\end{equation}
where we in the last inequality used that $4h_0<R\leq 1-5h_0$ and $\mathcal{M}_R>1$ along with $q\leq p$. Here $C=C(N,s,p,t,q,h_0)>0$ and $d=d(N,s,p,h_0)\geq 1$ is a constant.
Using \eqref{csuse1} together with the  properties of  $\eta$ {in \eqref{eta1p}}, we  obtain 
$$
\frac{|u(x)-u(y)|^{p-2}\,\left|\eta(x)^\frac{p}{2}-\eta(y)^\frac{p}{2}\right|^2}{|x-y|^{N+s\,p}}\leq C(N,s,p,t,q,h_0)(1+A)^d\,|x-y|^{-N+2\,(1-s)-\varepsilon\,(p-2)}.
$$
 Since the last exponent is strictly bigger than $-N$ by our choice of $\varepsilon$, we may conclude 
$$
\int_{B_R}\frac{|u(x)-u(y)|^{p-2}\,\left|\eta(x)^\frac{p}{2}-\eta(y)^\frac{p}{2}\right|^2}{|x-y|^{N+s\,p}} dy\leq C(N,s,p,t,q,h_0)(1+A)^d,
$$
for any $x\in B_R$.  This yields 
\begin{equation}\label{newI11ps1}
\begin{split}
|\mathcal{I}_{11,p,s,1}|&\leq C(1+A)^d\,\int_{B_R}\frac{|\delta_h u(x)|^{\beta+1}}{|h|^{1+\theta\,\beta}} dx\\
&\leq C(1+A)^d\,\|u\|_{L^\infty(B_{R+h_0})}\,\int_{B_R}\frac{|\delta_h u(x)|^{\beta}}{|h|^{1+\theta\,\beta}} dx\\
&\leq C(1+A)^d\,\int_{B_R}\frac{|\delta_h u(x)|^{\beta}}{|h|^{1+\theta\,\beta}} dx,\qquad \mbox{for some } C=C(N,s,p,t,q,h_0)>0.
\end{split}
\end{equation}
 Recalling \eqref{eq:I2original} we may also obtain the estimate 
\begin{equation}\label{newI12ps1}
|\mathcal{I}_{12,p,s,1}|\leq C\int_{B_R} \frac{|\delta_h u(x)|^{\beta-1+p}}{|h|^{1+\theta\,\beta}}\,dx\leq C\int_{B_R}\frac{|\delta_h u(x)|^{\beta}}{|h|^{1+\theta\,\beta}} dx,\qquad \mbox{ for some }C=C(N,s,p,t,q,h_0)>0,
\end{equation}
where we used that $\|u\|_{L^\infty(B_1)}\leq 1$.  Morever, the first inequality in \eqref{estI23p1} implies 
\begin{equation}\label{newI23ps1}
|\mathcal{I}_{2,p,s,1}|+|\mathcal{I}_{3,p,s,1}|\leq C\int_{B_R}\frac{|\delta_h u(x)|^{\beta}}{|h|^{1+\theta\,\beta}} dx, \qquad C=C(N,s,p,t,q,h_0)>0.
\end{equation}
Combining these new estimates \eqref{newI11ps1}, \eqref{newI12ps1}, \eqref{newI23ps1} with \eqref{estI1q1}, \eqref{estI2q1}, \eqref{estI3q1} and \eqref{estI1p1} in \eqref{Il1} and taking into account the estimates (4.12), (4.13) and (4.14) of \cite{BLS}, we obtain
$$
\sup_{0<|h|< h_0}\int_{B_r}\left|\frac{\delta^2_h u(x)}{|h|^\frac{1+s\,p+\theta\,\beta}{\beta-1+p}}\right|^{\beta-1+p}\,dx\leq C(1+A)^d\sup_{0<|h|< h_0}\int_{B_R}\left|\frac{\delta_h u(x)}{|h|^\frac{1+\theta\,\beta}{\beta}}\right|^\beta\,dx,
$$
where $C=C(N,s,p,t,q,\beta,h_0)>0$.
By using the second estimate in \cite[Lemma 2.6]{BLS} and using that
\[
\frac{1+\theta\,\beta}{\beta}<1,
\] 
 we may switch from first order differential quotients to second order differential quotients in the right-hand side, and arrive at
$$
\sup_{0<|h|< h_0}\int_{B_r}\left|\frac{\delta^2_h u(x)}{|h|^\frac{1+s\,p+\theta\,\beta}{\beta-1+p}}\right|^{\beta-1+p}\,dx\leq C(1+A)^d\left(\sup_{0<|h|< h_0}\int_{B_R+4h_0}\left|\frac{\delta^2_h u(x)}{|h|^\frac{1+\theta\,\beta}{\beta}}\right|^\beta\,dx+1\right),
$$
where $C=C(N,s,p,t,q,\beta,h_0)>0.$  Since $r=R-4\,h_0$, this implies the desired result.
\end{proof}

\subsection{Proof of the main result: Theorem \ref{teo:1higher}}
\begin{proof}
{We assume $x_0=0$. The proof is similar when $x_0\neq 0$.}  Again,  as in the proof of Theorem~\ref{teo:1}, it is enough to prove that for every $\e\in(0,\Gamma)$, there are constants $\sigma=\sigma(s,p,q,\e)$, $C=C(N,s,p,t,q,\e)>0$ and $\beta=\beta(N,s,p,\e)\geq 1$ such that 
\begin{equation}\label{mt-c1}
[u_r]_{C^{\Gamma-\e}(B_{\sigma})}\leq C(1+Ar^{sp-tq}\mathcal{M}_r^{q-p})^\beta,
\end{equation}
where
$$
u_r(x)=\frac{u(rx)}{\mathcal{M}_r}.
$$
For the rest of the proof, we omit the subscript $r$ and simply write $u$ instead of $u_r$.  From this point, the proof is identical to the proof of \cite[Theorem 5.2]{BLS} except that there are some additional constants that appear in the proof. For the sake of completeness, we present parts of the identical proof. 
For $i\in\mathbb{N}$, we define the sequences of exponents
$$
\beta_0=p,\qquad\qquad  \beta_{i+1}=\beta_i+p-1, 
$$
and
$$
\theta_0=s-\frac{1}{p},\qquad \theta_{i+1}=\frac{\theta_i\,\beta_i+s\,p}{\beta_{i+1}}=\theta_i\,\frac{p+i\,(p-1)}{p+(i+1)(p-1)}+\frac{s\,p}{p+(i+1)(p-1)}.
$$
By induction, we see that $\{\theta_i\}_{i\in\mathbb{N}}$ is explicitly given by the increasing sequence
$$
 \theta_i =\left(s-\frac{1}{p}\right)\,\frac{p}{p+i\,(p-1)}+\frac{s\,p\,i}{p+i\,(p-1)},
$$
and thus
$$
\lim_{i\to\infty} \theta_i = \frac{s\,p}{p-1}.
$$
The proof is now split into two different cases.
\vskip.2cm\noindent
{\bf  Case 1: $s\,p\leq (p-1)$.} 
Fix $0<\e<\frac{sp}{p-1}$ and define $\delta=\frac{sp}{p-1}-\e$. Further, choose $j_\infty\in\mathbb{N}\setminus\{0\}$ such that
\[
\delta<\frac{1}{\beta_{j_\infty}}+\theta_{j_\infty}-\frac{N}{\beta_{j_\infty}}.
\]
This is clearly possible since 
\[
\lim_{j\to\infty} \beta_j=\infty,\qquad \lim_{j\to\infty}\theta_j=\frac{s\,p}{p-1}\qquad \mbox{ and } \qquad \delta< \frac{s\,p}{p-1}.
\]
Define also as in the proof of Theorem \ref{teo:1}
$$
h_0=\frac{1}{64\,j_\infty},\qquad R_j=\frac{7}{8}-4\,(2\,j+1)\,h_0=\frac{7}{8}-\frac{2\,j+1}{16\,j_\infty},\qquad \mbox{ for } j=0,\dots,j_\infty.
$$
We note that 
\[
R_0+4\,h_0=\frac{7}{8}\qquad \mbox{ and }\qquad R_{j_\infty-1}-4\,h_0=\frac{3}{4}.
\] 
By applying Proposition\footnote{Note that in this case we will always have $1+\theta_j\beta_j<\beta_j$, so that Proposition \ref{prop:improve2I} is applicable.}{\ref{prop:improve2I}}
 with
\[
R=R_j, \qquad \theta=\theta_j \qquad \mbox{ and }\qquad \beta=\beta_j,\quad \mbox{ for } j=0,\ldots,j_\infty-1,
\] 
and observing that $R_j-4\,h_0=R_{j+1}+4\,h_0$ and that by construction
$$
\frac{1+s\,p+\theta_j\,\beta_j}{\beta_j+(p-1)}=\frac{1+\theta_{j+1}\,\beta_{j+1}}{\beta_{j+1}},
$$
we obtain the iterative scheme of estimates
\[
\sup\limits_{0<|h|< h_0}\left\|\dfrac{\delta^2_h u}{|h|^{\frac{1+\theta_1\beta_1}{\beta_1}}}\right\|_{L^{\beta_1}(B_{R_1+4h_0})}\\
\leq  C(Ar^{sp-tq}\mathcal{M}_r^{q-p}+1)^d\,\sup\limits_{0<|h|< h_0}\left(\left\|\dfrac{\delta^2_h u }{|h|^s}\right\|_{L^{p}(B_{\frac{7}{8}})}+1\right),
\]
and
\[\begin{split}
&\sup\limits_{0<|h|< h_0}\left\|\dfrac{\delta^2_h u}{|h|^{\frac{1+\theta_{j+1}\beta_{j+1}}{\beta_{j+1}}}}\right\|_{L^{\beta_{j+1}}(B_{R_{j+1}+4h_0})}
\\
&\leq  C(Ar^{sp-tq}\mathcal{M}_r^{q-p}+1)^d\,\sup\limits_{0<|h|< h_0}\left(\left\|\dfrac{\delta^2_h u }{|h|^{\frac{1+\theta_{j}\beta_{j}}{\beta_{j}}}}\right\|_{L^{\beta_i}(B_{R_j+4\,h_0})}+1\right),
\end{split}
\]
for $j=1,\ldots,j_\infty-2$, and finally
\[\begin{split}
&\sup_{0<|h|< {h_0}}\left\|\frac{\delta^2_h u}{|h|^{\frac{1}{\beta_{j_\infty}}+\theta_{j_\infty}}}\right\|_{L^{\beta_{j_\infty}}(B_{\frac{3}{4}})}\\&\leq C(Ar^{sp-tq}\mathcal{M}_r^{q-p}+1)^d \sup_{0<|h|< h_0}\left(\left\|\frac{\delta^2_h u }{|h|^{\frac{1+\theta_{j_\infty-1}\beta_{j_\infty-1}}{\beta_{j_\infty-1}}}}\right\|_{L^{\beta_{j_\infty-1}}(B_{R_{j_\infty-1}+4\,h_0})}+1\right).
\end{split}
\]
As in \eqref{eq:1sttofrac} we have
$$
\sup_{0<|h|< h_0}\left\|\frac{\delta^2_h u }{|h|^s}\right\|_{L^{p}(B_{\frac{7}{8}})}\leq C(Ar^{sp-tq}\mathcal{M}_r^{q-p}+1),\nonumber 
$$
where $C=C(N,s,p,t,q,\e)>0$.
Hence, the previous iterative scheme of inequalities implies
\begin{equation}
\label{differentiability}
\sup_{0<|h|< {h_0}}\left\|\frac{\delta^2_h u}{|h|^{\frac{1}{\beta_{j_\infty}}+\theta_{j_\infty}}}\right\|_{L^{\beta_{j_\infty}}(B_{\frac{3}{4}})}\leq C(Ar^{sp-tq}\mathcal{M}_r^{q-p}+1)^d,
\end{equation}
where
$C=C(N,s,p,t,q,\e)>0$ and $d=d(N,s,p,\e)\geq 1$ are some constants.
From here we may repeat the arguments at the end of the proof of Theorem~\ref{teo:1} ({see \eqref{Neq-c1} and \eqref{Neq2-c1}}) and use Theorem~\ref{emb2} with $\beta = \frac{1}{\beta_{j_\infty}}+\theta_{j_\infty}$, $q=\beta_{j_\infty}$ and $\alpha =\delta=\Gamma-\e=\frac{sp}{p-1}-\e$ to obtain 
\begin{equation}\label{cov1}
[u]_{C^{\frac{sp}{p-1}-\e}(B_{\frac{1}{2}})}\leq C(Ar^{sp-tq}\mathcal{M}_r^{q-p}+1)^d,
\end{equation}
where $C=C(N,s,p,t,q,\e)>0$ and $d=d(N,s,p,\e)\geq 1$ are some constants. This is \eqref{mt-c1} for $\sigma=\frac{1}{2}$.
\vskip.2cm\noindent
{\bf Case 2: $s\,p> (p-1)$.} Fix $0<\e<1$ and let $\delta=1-\e$. Let $k_\infty\in\mathbb{N}\setminus\{0\}$ be such that 
$$
\frac{1+\theta_{k_\infty-1}\,\beta_{k_\infty-1}}{\beta_{k_\infty-1}}< 1\qquad \mbox{ and }\qquad \frac{1+\theta_{k_\infty}\,\beta_{k_\infty}}{\beta_{k_\infty}}\ge 1.
$$
Observe that such a choice is feasible, since 
\[
\lim_{k\to\infty} \frac{1+\theta_k\,\beta_k}{\beta_k}=\frac{s\,p}{p-1}>1.
\] 
Now choose $l_\infty\in\mathbb{N}\setminus\{0\}$ so that 
$$
\delta<1-\frac{N}{k_\infty+l_\infty},
$$
and let 
\[
\gamma=1-\gamma',\qquad \mbox{ for some } 0<\gamma'<1 \mbox{ such that } \delta<1-\gamma'-\frac{N}{k_\infty+l_\infty}.
\] 
Define also 
$$
h_0=\frac{1}{64\,(k_\infty+l_\infty)},\qquad R_i=\frac{7}{8}-4\,(2\,i+1)\,h_0=\frac{7}{8}-\frac{2\,i+1}{16\,(k_\infty+l_\infty)},\qquad \mbox{ for } i=0,\dots,k_\infty+l_\infty.
$$
We note that 
\[
R_0+4\,h_0=\frac{7}{8}\qquad \mbox{ and }\qquad R_{(k_\infty+l_\infty)-1}-4\,h_0=\frac{3}{4}.
\] 
By applying Proposition\footnote{Note that for $i\leq k_\infty$ we have $1+\theta_i\,\beta_i<\beta_i$, so that Proposition \ref{prop:improve2I} applies.} {\ref{prop:improve2I}} with
\[
R=R_i, \qquad \theta=\theta_i \qquad \mbox{ and }\qquad \beta=\beta_i,\quad \mbox{ for } i=0,\ldots,k_\infty-1,
\] 
and observing that $R_i-4\,h_0=R_{i+1}+4\,h_0$ and that
$$
\frac{1+s\,p+\theta_i\,\beta_i}{\beta_i+(p-1)}=\frac{1+\theta_{i+1}\,\beta_{i+1}}{\beta_{i+1}},
$$
we arrive as in {\bf Case 1} at
\begin{equation}\label{ncase1}
\sup_{0<|h|< {h_0}}\left\|\frac{\delta^2_h u}{|h|^{\gamma}}\right\|_{L^{\beta_{k_\infty}}(B_{R_{k_\infty}+4h_0})}\leq \sup_{0<|h|< {h_0}}\left\|\frac{\delta^2_h u}{|h|^{\frac{1}{\beta_{k_\infty}}+\theta_{k_\infty}}}\right\|_{L^{\beta_{k_\infty}}(B_{R_{k_\infty}+4h_0})}\leq C(Ar^{sp-tq}\mathcal{M}_r^{q-p}+1)^d,
\end{equation}
where we used that $\gamma<1\le \frac{1}{\beta_{k_\infty}}+\theta_{k_\infty}$. Here $C=C(N,s,p,t,q,\e)>0$ and $d=d(N,s,p,t,q,\e)\geq 1$ are some constants. We now apply Proposition {\ref{prop:improve2I}} with
\[
R=R_i, \qquad \beta=\beta_i\qquad \mbox{ and }\qquad \theta=\widetilde \theta_i=\gamma-\frac{1}{\beta_i} \qquad  \mbox{ for } i=k_\infty,\ldots,k_\infty+l_\infty-1.
\] 
Observe that by construction we have
$$
\frac{1+\widetilde \theta_i\,\beta_i}{\beta_i}=\gamma,\qquad \mbox{ for } i=k_\infty,\ldots,k_\infty+l_\infty-1,
$$
and using that $s\,p>(p-1)$
$$
\frac{1+s\,p+\widetilde \theta_i\, \beta_i}{\beta_i+p-1}>\frac{p+\widetilde \theta_i\, \beta_i}{\beta_i+p-1}=1+\frac{\beta_i\,(\gamma-1)}{\beta_i+p-1}>\gamma,\quad \mbox{ for } i=k_\infty,\ldots,k_\infty+l_\infty-1.
$$
We obtain the following chain of estimates for $i=k_\infty,\ldots,k_\infty+l_\infty-2$,
$$
\sup_{0<|h|<h_0}\left\|\dfrac{\delta^2_h u}{|h|^{\gamma}}\right\|_{L^{\beta_{i+1}}(B_{R_{i+1}+4h_0})}\leq C(Ar^{sp-tq}\mathcal{M}_r^{q-p}+1)^d\,\sup\limits_{0<|h|< h_0}\left(\left\|\dfrac{\delta^2_h u }{|h|^{\gamma}}\right\|_{L^{\beta_i}(B_{R_i+4h_0})}+1\right),\,
$$
and finally
$$
\sup_{0<|h|< {h_0}}\left\|\frac{\delta^2_h u}{|h|^{\gamma}}\right\|_{L^{\beta_{k_\infty+l_\infty}}(B_{\frac{3}{4}})}\leq C(Ar^{sp-tq}\mathcal{M}_r^{q-p}+1)^d\sup_{0<|h|< h_0}\left(\left\|\frac{\delta^2_h u }{|h|^{\gamma}}\right\|_{L^{\beta_{k_\infty+l_\infty-1}}(B_{R_{k_\infty+l_\infty-1}+4h_0})}+1\right),
$$
where $C=C(N,s,p,t,q,\e)>0$.
Hence, recalling that $\gamma=1-\gamma'$, we conclude
$$
\sup_{0<|h|< {h_0}}\left\|\frac{\delta^2_h u}{|h|^{1-\gamma'}}\right\|_{L^{\beta_{k_\infty+l_\infty}}(B_{\frac{3}{4}})}\leq C(Ar^{sp-tq}\mathcal{M}_r^{q-p}+1)^d,
$$
where we also used \eqref{ncase1}. Here $C=C(N,s,p,t,q,\e)>0$ and $d=d(N,s,p,\e)\geq 1$ are some constants. We are again in the position to repeat the arguments at the end of the proof of Theorem \ref{teo:1} ({see \eqref{Neq-c1} and \eqref{Neq2-c1}}) and use Theorem~\ref{emb2} with $\beta = 1-\gamma'$ , $q=\beta_{i_\infty+j_\infty}$ and $\alpha =\delta=1-\e$ to obtain  
\begin{equation}\label{cov2}
[u]_{C^{1-\e}(B_{\frac{1}{2}})}\leq C(Ar^{sp-tq}\mathcal{M}_r^{q-p}+1)^d,
\end{equation}
where $C=C(N,s,p,t,q,\e)>0$ and $d=d(N,s,p,\e)\geq 1$ are constants.
This is \eqref{mt-c1} for $\sigma=\frac{1}{2}$.
This concludes the proof.
\end{proof}

\section{Case II: $1<p\leq 2\leq q$}\label{sec:II}
Throughout this section, we assume that $1<p\leq 2\leq q$ and $s,t\in(0,1)$ satisfy $\frac{sp}{p-1}\geq \frac{tq}{q-1}$ unless stated otherwise.
\subsection{$t$-regularity using both the superquadratic and subquadratic schemes}
 In this subsection we present the following two versions of improved $t$-Besov regularity using either the superquadratic scheme or the subquadratic scheme.
\begin{prop}
\label{prop:improvesuper} (Improved $t$-Besov regularity using the superquadratic scheme) Assume $A>0$ and let $u\in X(B_2)$ be a local weak solution of \eqref{eq:maineqnew} in $B_2$  satisfying 
\begin{equation}
\label{eq2tsuper:bounds}
\|u\|_{L^\infty(B_1)}\leq 1\qquad \mbox{ and }\qquad  \int_{\mathbb{R}^N\setminus B_1} \frac{|u(y)|^{l-1}}{|y|^{N+m\,l}}\,  dy\leq 1,
\end{equation}
for $(l,m)\in {(p,s),(t,q)}$.
Then for any $q\leq\mu <\infty$, $0<h_0<\frac{1}{10}$ and $R$ such that $4\,h_0<R\le 1-5\,h_0$,  
we have
\begin{equation*}
\begin{split}
\sup_{0<|h|< h_0}\left\|\frac{\delta^2_h u}{|h|^{t}}\right\|_{L^{\mu+1}(B_{R-4\,h_0})}^{{\mu}+1}&\leq C\Big(A^{-1}+1\Big)\,\left(\sup_{0<|h|< h_0}\left\|\frac{\delta^2_h u }{|h|^t}\right\|_{L^{\mu}(B_{R+4h_0})}^{\mu}+1\right),
\end{split}
\end{equation*}
for some positive constant $C=C(N,s,p,t,q,\mu,h_0)$.
\end{prop}
\begin{proof}
 Define 
$
r=R-4h_0$ and recall $\quad d\mu_{l,m}=\frac{dx dy}{|x-y|^{N+lm}}.$ Take $\varphi\in W^{s,p}(B_R)\cap W^{t,q}(B_R)$ vanishing outside $B_{\frac{R+r}{2}}$ and  $0<|h|<h_0$. Testing \eqref{wksol} with $\varphi$ and $\varphi_{-h}$ and performing a change of variable yields
\begin{equation}
\label{eq2tsuper:differentiated}
\frac1h \sum_{(l,m,k)\in\{(p,s,1),(q,t,A)\}}\int_{\mathbb{R}^N} \int_{\mathbb{R}^N}k \Big(J_l(u_h(x)-u_h(y))-J_l(u(x)-u(y))\Big)\,\Big(\varphi(x)-\varphi(y)\Big)\,d\mu_{l,m}=0.
\end{equation}
 Take  $\eta\in C^\infty_0(B_R)$  such that
\begin{equation}\label{eq2tsuper:etap}
0\le \eta\le 1\text{ in }B_R,\qquad\eta\equiv 1\text{ in }B_r,\qquad \eta\equiv 0 \mbox{ in } \mathbb{R}^N\setminus B_{\frac{R+r}{2}},\qquad |\nabla \eta|\le \frac{C}{R-r}=\frac{C}{4\,h_0},
\end{equation}
where $C=C(N)>0$.  Inserting 
\[
\varphi=J_{\beta+1}\left(\frac{\delta_h u}{|h|^{\theta}}\right)\,\eta^q,
\]
in \eqref{eq2tsuper:differentiated} with $\beta\geq 2$ and $1+\theta\beta\geq 0$,
we obtain
\begin{equation*}
\begin{split}
&\sum_{(l,m,k)\in\{(p,s,1),(q,t,A)\}}\iint_{\mathbb{R}^N\times\mathbb{R}^N}k \frac{\Big(J_l(u_h(x)-u_h(y))-J_l(u(x)-u(y))\Big)}{|h|^{1+\theta\beta}}\\
&\times\Big(J_{\beta+1}(u_h(x)-u(x))\,\eta(x)^q-J_{\beta+1}(u_h(y)-u(y))\,\eta(y)^q\Big)\,d\mu_{l,m}=0,
\end{split}
\end{equation*}
that is
\begin{equation}
\label{eq2tsuper:Il}
\sum_{(l,m,k)\in\{(p,s,1),(q,t,A)\}}k(\mathcal{I}_{1,l,m}+\mathcal{I}_{2,l,m}+\mathcal{I}_{3,l,m})=0,
\end{equation}
where
\[
\begin{split}
\mathcal{I}_{1,l,m,k}:=\iint_{B_R\times B_R}&k \frac{\Big(J_l(u_h(x)-u_h(y))-J_l(u(x)-u(y))\Big)}{|h|^{1+\theta\beta}}\\
&\times\Big(J_{\beta+1}(u_h(x)-u(x))\,\eta(x)^q-J_{\beta+1}(u_h(y)-u(y))\,\eta(y)^q\Big)d\mu_{l,m},
\end{split}
\]
\[
\begin{split}
\mathcal{I}_{2,l,m,k}:=\iint_{B_\frac{R+r}{2}\times (\mathbb{R}^N\setminus B_R)}& k\frac{\Big(J_l(u_h(x)-u_h(y))-J_l(u(x)-u(y))\Big)}{|h|^{1+\theta\,\beta}}\\
&\times J_{\beta+1}(u_h(x)-u(x))\,\eta(x)^q\,d\mu_{l,m},
\end{split}
\]
and
\[
\begin{split}
\mathcal{I}_{3,l,m,k}:=-\iint_{(\mathbb{R}^N\setminus B_R)\times B_\frac{R+r}{2}}& k\frac{\Big(J_l(u_h(x)-u_h(y))-J_l(u(x)-u(y))\Big)}{|h|^{1+\theta\,\beta}}\\
&\times J_{\beta+1}(u_h(y)-u(y))\,\eta(y)^q\,d\mu_{l,m},
\end{split}
\]
where we used that $\eta$ vanishes identically outside $B_{\frac{R+r}{2}}$.

We estimate $\mathcal{I}_{j,l,m,k}$ for $j=1,2,3$ and $(l,m,k)\in\{(p,s,1),(q,t,A)\}$ separately. First we estimate for $(l,m,k)=(p,s,1)$. We remark that $\eta^q=(\eta^\frac{q}{2})^2$ and that $\eta^\frac{q}{2}$ is Lipschitz and enjoys the same properties as $\eta$, including \eqref{eq2tsuper:etap}. Since $p\leq 2$, taking into account \eqref{eq2tsuper:bounds}, we  may proceed as in the estimates of (3.15), (3.17) and (3.18) in \cite{GL} to obtain
\begin{equation}\label{eq2tsuper:estI1q}
\begin{split}
\mathcal{I}_{1,p,s,1}
&\geq
-C\int_{B_R} \frac{|\delta_h u(x)|^{p+\beta-1}}{|h|^{1+\theta\,\beta}} dx\geq -C\int_{B_R} \frac{|\delta_h u(x)|^{\beta}}{|h|^{1+\theta\,\beta}} dx,
\end{split}
\end{equation}
since $|u|\leq 1$ in $B_1$.
Here $C=C(N,s,p,\beta,h_0)>0$. Further, we get
\begin{equation}
\label{eq2tsuper:estI2q}
|\mathcal{I}_{2,p,s,1}|\leq C\int_{B_R}\frac{|\delta_h u(x)|^{\beta}}{|h|^{1+\theta\,\beta}} dx,
\end{equation}
and
\begin{equation}
\label{eq2tsuper:estI3q}
|\mathcal{I}_{3,p,s,1}|\leq C\int_{B_R}\frac{|\delta_h u(x)|^{\beta}}{|h|^{1+\theta\,\beta}} dx.
\end{equation}
Here $C=C(N,s,p,h_0)>0$ is a constant.

Next, we estimate $\mathcal{I}_{j,l,m,A}$ for $j=1,2,3$ and $(l,m)=(q,t)$. Again, since $q\geq 2$, taking into account the properties of $\eta$ and proceeding along the lines of the proof of the estimates of $\mathcal{I}_1, \mathcal{I}_{11}$ and $\mathcal{I}_{12}$ from pages 814-819 in \cite{BLS}, we get
\begin{equation}\label{eq2tsuper:estI1p}
\begin{split}
\mathcal{I}_{1,q,t,A}\ge cA& \left[\frac{|\delta_h u|^\frac{\beta-1}{q}\,\delta_h u}{|h|^\frac{1+\theta\,\beta}{q}}\,\eta\right]^q_{W^{t,q}(B_R)}-CA\,\mathcal{I}_{11}-CA\,\mathcal{I}_{12}
\end{split}
\end{equation}
where $c=c(q,\beta)>0$ and $C=C(q,\beta)>0$, where we set 
\begin{equation*}
\begin{split}
\mathcal{I}_{11}:=\,\iint_{B_R\times B_R} &\left(|u_h(x)-u_h(y)|^\frac{q-2}{2}+|u(x)-u(y)|^\frac{q-2}{2}\right)^2\\
&\times \left|\eta(x)^\frac{q}{2}-\eta(y)^\frac{q}{2}\right|^2\, \frac{|\delta_h u(x)|^{\beta+1}+|\delta_h u(y)|^{\beta+1}}{|h|^{1+\theta\,\beta}}\,d\mu_{q,t},
\end{split}
\end{equation*}
and
\[
\begin{split}
\mathcal{I}_{12}&:=\,\iint_{B_R\times B_R}\, \left(\frac{|\delta_h u(x)|^{\beta-1+q}}{|h|^{1+\theta\,\beta}}+\frac{|\delta_h u(y)|^{\beta-1+q}}{|h|^{1+\theta\,\beta}}\right)\, |\eta(x)-\eta(y)|^q\,d\mu_{q,t}
\end{split}
\]
with the estimates
\begin{equation}\label{eq2tsuper:estI11}
|\mathcal{I}_{11}|\leq C\,\left(\int_{B_R}\frac{|\delta_h u(x)|^{\frac{\beta\,\mu}{\mu-q+2}}}{|h|^{(1+\theta\,\beta)\frac{\mu}{\mu-q+2}}}\, dx +\sup_{0<|h|< h_0}\left\|\frac{\delta^2_h u}{|h|^t}\right\|_{L^{\mu}(B_{R+4h_0})}^{\mu}+1\right),
\end{equation}
where $C=C(N,t,q,\mu,h_0)>0$ and
\begin{equation}\label{eq2tsuper:estI12}
|\mathcal{I}_{12}|\leq C\left( \int_{B_R} \frac{|\delta_h u|^{\frac{\beta \mu}{\mu-q+2}}}{|h|^{(1+\theta\,\beta)\,\frac{\mu}{\mu-q+2}}}\,dx+1\right),\qquad\mbox{ with } C=C(N,t,q,\mu,h_0)>0.
\end{equation}
Now exactly as in the estimate (4.9) in \cite{BLS}, we obtain
\begin{equation}\label{eq2tsuper:estI23p}
\begin{split}
(|\mathcal{I}_{2,q,t,A}|+|\mathcal{I}_{3,q,t}|)&\leq CA\,\int_{B_{\frac{R+r}{2}}}\frac{|\delta_h u|^{\beta}}{|h|^{1+\theta\beta}} dx\leq CA\,\left(1+\int_{B_{R}}\left|\frac{\delta_h u}{|h|^{\frac{1+\theta\beta}{\beta}}}\right|^{\frac{\beta\,\mu}{\mu-q+2}}\,dx\right)
\end{split}
\end{equation}
where $C=C(N,t,q,h_0)>0$. By inserting the estimates \eqref{eq2tsuper:estI1q}, \eqref{eq2tsuper:estI2q}, \eqref{eq2tsuper:estI3q}, \eqref{eq2tsuper:estI1p}, \eqref{eq2tsuper:estI11}, \eqref{eq2tsuper:estI12} and \eqref{eq2tsuper:estI23p} in \eqref{eq2tsuper:Il}, we obtain
\begin{equation}\label{eq2tsuper:estIl}
\begin{split}
&
\left[\frac{|\delta_h u|^\frac{\beta-1}{q}\,\delta_h u}{|h|^\frac{1+\theta\,\beta}{q}}\,\eta\right]^q_{W^{t,q}(B_R)}\\
&
\leq  \frac{C}{A}\int_{B_R}\frac{|\delta_h u(x)|^{\beta}}{|h|^{1+\theta\,\beta}}dx+C\,\left(\int_{B_{R}}\left|\frac{\delta_h u}{|h|^{\frac{1+\theta\beta}{\beta}}}\right|^{\frac{\beta\, \mu}{\mu-q+2}}\,dx+\sup_{0<|h|< h_0}\left\|\frac{\delta^2_h u}{|h|^t}\right\|_{L^\mu(B_{R+4h_0})}^{\mu}+1\right)\\
&\leq C\Big(A^{-1}+1\Big)\int_{B_{R}}\left|\frac{\delta_h u}{|h|^{\frac{1+\theta\beta}{\beta}}}\right|^{\frac{\beta\, \mu}{\mu-q+2}}\,dx+C\left(\sup_{0<|h|< h_0}\left\|\frac{\delta^2_h u}{|h|^t}\right\|_{L^\mu(B_{R+4h_0})}^{\mu}+1\right)
\end{split}
\end{equation}
where $C=C(N,s,p,t,q,\beta,\mu,h_0)>0$. Here we also used Young's inequality. Now, we estimate the term in the L.H.S. of \eqref{eq2tsuper:estIl} exactly as in \cite[estimate of (4.15), page 822]{BLS} and combine the resulting estimate in \eqref{eq2tsuper:estIl} above to obtain
\begin{equation*}
\begin{split}
&\sup_{0<|h|< h_0}\int_{B_r}\left|\frac{\delta^2_h u}{|h|^\frac{1+tq+\theta\,\beta}{\beta-1+q}}\right|^{\beta-1+q}\,dx\\
&\leq C\Big(A^{-1}+1\Big)\,\int_{B_{R}}\left|\frac{\delta_h u}{|h|^{\frac{1+\theta\beta}{\beta}}}\right|^{\frac{\beta\,\mu}{\mu-q+2}}\,dx+C\left(\sup_{0<|h|< h_0}\left\|\frac{\delta^2_h u}{|h|^t}\right\|_{L^{\mu}(B_{R+4h_0})}^{\mu}+1\right),
\end{split}
\end{equation*}
for $C=C(N,s,p,t,q,\beta,\mu,h_0)>0$, where the last inequality above follows by using that $|u|\leq 1$ in $B_1$ and $q\geq 2$. 

Next, we choose $\theta$ so that $\frac{1+\theta\,\beta}{\beta}=t< 1$ and proceed along the lines of the proof of \cite[estimate of (4.17), page 822]{BLS} to obtain
\begin{equation*}
\begin{split}
&\sup_{0<|h|< h_0}\int_{B_r}\left|\frac{\delta^2_h u}{|h|^\frac{ 1+t\,q+\theta\,\beta }{\beta-1+q}}\right|^{\beta-1+q}\,dx\\
&\leq C\Big(A^{-1}+1\Big)\,\sup_{0<|h|< h_0}\left\|\frac{\delta^2_h u }{|h|^t}\right\|_{L^\frac{\beta\mu}{\mu-q+2}(B_{R+4h_0})}^\frac{\beta\mu}{\mu-q+2}+C\left(\sup_{0<|h|< h_0}\left\|\frac{\delta^2_h u}{|h|^t}\right\|_{L^{\mu}(B_{R+4\,h_0})}^{\mu}+ 1\right),
\end{split}
\end{equation*}
for some constant $C=C(N,s,p,t,q,\beta,\mu,h_0)>0$. Now, as in Step 5 of \cite{BLS} choosing $\beta=\mu-q+2$ and further arguing exactly same as in Step 5 of \cite{BLS}, we obtain
$$
\sup_{0<|h|< h_0}\left\|\frac{\delta^2_h u}{|h|^{t}}\right\|_{L^{\mu+1}(B_{R-4\,h_0})}^{\mu+1}\leq C\Big(A^{-1}+1\Big)\,\left(\sup_{0<|h|< h_0}\left\|\frac{\delta^2_h u }{|h|^t}\right\|_{L^{\mu}(B_{R+4h_0})}^{\mu}+1\right),
$$
for $C=C(N,s,p,t,q,\mu,h_0)>0$.
\end{proof}

\begin{prop}
\label{prop:improvesub} (Improved $t$-Besov regularity using the subquadratic scheme)
Assume that $sp\geq tq$ and $u\in X(B_2)$ is a local weak solution of equation \eqref{eq:maineqnew} in $B_2$  such that
\begin{equation}
\label{eq2tsub:bounds}
\|u\|_{L^\infty(B_1)}\leq 1\qquad \mbox{ and }\qquad  \int_{\mathbb{R}^N\setminus B_1} \frac{|u(y)|^{l-1}}{|y|^{N+m\,l}}\,  dy\leq 1,
\end{equation} 
for $(l,m)\in \{(p,s),(t,q)\}$. 
Then for any $q\leq\mu <\infty$, $0<h_0<\frac{1}{10}$, we have $R$ such that $4\,h_0<R\le 1-5\,h_0$,  
we have 
\begin{equation*}
\begin{split}
\sup_{0<|h|< h_0}\left\|\frac{\delta^2_h u}{|h|^t}\right\|_{L^{\mu+2}(B_{R-4\,h_0})}^{{\mu}+2}
\leq C\Big({A}+1\Big)\,\left(\sup_{0<|h|< h_0}\left\|\frac{\delta^2_h u }{|h|^t}\right\|_{L^{\mu}(B_{R+4h_0})}^{\mu}+1\right),
\end{split}
\end{equation*}
for some positive constant $C=C(N,s,p,t,q,\mu,h_0)$.
\end{prop}

\begin{proof} The setup and the starting point is the same as in the proof of Proposition \ref{prop:improvesuper}. We also arrive at
\begin{equation*}
\begin{split}
&\sum_{(l,m,k)\in\{(p,s,1),(q,t,A)\}}k\iint_{\mathbb{R}^N\times\mathbb{R}^N}  \frac{\Big(J_l(u_h(x)-u_h(y))-J_l(u(x)-u(y))\Big)}{|h|^{1+\theta\beta}}\\
&\times\Big(J_{\beta+1}(u_h(x)-u(x))\,\eta(x)^q-J_{\beta+1}(u_h(y)-u(y))\,\eta(y)^q\Big)\,d\mu_{l,m}=0,
\end{split}
\end{equation*}
that is 
\begin{equation}
\label{eq2tsub:Il}
\sum_{(l,m,k)\in\{(p,s,1),(q,t,A)\}}(\mathcal{I}_{1,l,m,k}+\mathcal{I}_{2,l,m,k}+\mathcal{I}_{3,l,m,k})=0,
\end{equation}
where
\[
\begin{split}
\mathcal{I}_{1,l,m,k}:=k\iint_{B_R\times B_R}& \frac{\Big(J_l(u_h(x)-u_h(y))-J_l(u(x)-u(y))\Big)}{|h|^{1+\theta\beta}}\\
&\times\Big(J_{\beta+1}(u_h(x)-u(x))\,\eta(x)^q-J_{\beta+1}(u_h(y)-u(y))\,\eta(y)^q\Big)d\mu_{l,m},
\end{split}
\]
\[
\begin{split}
\mathcal{I}_{2,l,m,k}:=k\iint_{B_\frac{R+r}{2}\times (\mathbb{R}^N\setminus B_R)}& \frac{\Big(J_l(u_h(x)-u_h(y))-J_l(u(x)-u(y))\Big)}{|h|^{1+\theta\,\beta}}\\
&\times J_{\beta+1}(u_h(x)-u(x))\,\eta(x)^q\,d\mu_{l,m},
\end{split}
\]
and
\[
\begin{split}
\mathcal{I}_{3,l,m,k}:=-k\iint_{(\mathbb{R}^N\setminus B_R)\times B_\frac{R+r}{2}}& \frac{\Big(J_l(u_h(x)-u_h(y))-J_l(u(x)-u(y))\Big)}{|h|^{1+\theta\,\beta}}\\
&\times J_{\beta+1}(u_h(y)-u(y))\,\eta(y)^q\,d\mu_{l,m},
\end{split}
\]
where we used that $\eta$ vanishes identically outside $B_{\frac{R+r}{2}}$.

We estimate $\mathcal{I}_{j,l,m,k}$ for $j=1,2,3$ and $(l,m)\in\{(p,s,1),(q,t,A)\}$ separately, beginning with the case $(l,m)=(p,s)$. To this end, we remark that $\eta^q=(\eta^\frac{q}{2})^2$ and that $\eta^\frac{q}{2}$ is Lipschitz and enjoys the same properties as $\eta$. Since $p\leq 2$, and \eqref{eq2tsub:bounds} holds, we may proceed along the lines of the proof of the estimates of $\mathcal{I}_1, \mathcal{I}_2$ and $\mathcal{I}_3$ in \cite{GL} (see pages 5767-5768 and estimates (3.17), (3.18)) and obtain for $2\sigma=sp$ that  
\begin{equation}\label{eq2tsub:estI1q}
\begin{split}
\mathcal{I}_{1,p,s,1}
&\geq {c}\left[\frac{|\delta_h u|^\frac{\beta-1}{2}\,\delta_h u}{|h|^\frac{1+\theta\,\beta}{2}}\,\eta^{\frac{q}{2}}\right]^2_{W^{\sigma,2}(B_R)}-C\int_{B_R} \frac{|\delta_h u(x)|^{p+\beta-1}}{|h|^{1+\theta\,\beta}} dx\\
& -C\,\int_{B_R}\, \frac{|\delta_h u(x)|^{\beta+1}}{|h|^{1+\theta\,\beta}}dx.
\end{split}
\end{equation}
Here $c=c(p,\beta)>0$ and $C=C(N,s,p,\beta,h_0)>0$. In addition, 
\begin{equation}
\label{eq2tsub:estI2q}
|\mathcal{I}_{2,p,s,1}|\leq C\int_{B_R}\frac{|\delta_h u(x)|^{\beta}}{|h|^{1+\theta\,\beta}} dx\leq C\left( \int_{B_R} \frac{|\delta_h u|^{\frac{\beta \mu}{\mu-q+2}}}{|h|^{(1+\theta\,\beta)\,\frac{\mu}{\mu-q+2}}}\,dx+1\right),
\end{equation}
and
\begin{equation}
\label{eq2tsub:estI3q}
|\mathcal{I}_{3,p,s,1}|\leq C\int_{B_R}\frac{|\delta_h u(x)|^{\beta}}{|h|^{1+\theta\,\beta}}\leq  C\left( \int_{B_R} \frac{|\delta_h u|^{\frac{\beta \mu}{\mu-q+2}}}{|h|^{(1+\theta\,\beta)\,\frac{\mu}{\mu-q+2}}}\,dx+1\right) dx.
\end{equation}
Here  $C=C(N,s,p,h_0)>0$ is a constant.

Next, we estimate $\mathcal{I}_{j,l,m,A}$ for $j=1,2,3$ and $(l,m)=(q,t)$. Again, since $q\geq 2$, we proceed as in Proposition \ref{prop:improvesuper} and obtain
\begin{equation}\label{eq2tsub:estI1p}
\begin{split}
\mathcal{I}_{1,q,t,A}\ge cA& \left[\frac{|\delta_h u|^\frac{\beta-1}{q}\,\delta_h u}{|h|^\frac{1+\theta\,\beta}{q}}\,\eta\right]^q_{W^{t,q}(B_R)}-CA\,\mathcal{I}_{11}-CA\,\mathcal{I}_{12}
\end{split}
\end{equation}
where $c=c(q,\beta)>0$, $C=C(q,\beta)>0$. In addition, 

\begin{equation}\label{eq2tsub:estI11}
|\mathcal{I}_{11}|\leq C\,\left(\int_{B_R}\frac{|\delta_h u(x)|^{\frac{\beta\,\mu}{\mu-q+2}}}{|h|^{(1+\theta\,\beta)\frac{\mu}{\mu-q+2}}}\, dx +\sup_{0<|h|< h_0}\left\|\frac{\delta^2_h u}{|h|^t}\right\|_{L^{\mu}(B_{R+4h_0})}^{\mu}+1\right),
\end{equation}
where $C=C(N,t,q,\mu,h_0)>0$ and
\begin{equation}\label{eq2tsub:estI12}
|\mathcal{I}_{12}|\leq C\left( \int_{B_R} \frac{|\delta_h u|^{\frac{\beta \mu}{\mu-q+2}}}{|h|^{(1+\theta\,\beta)\,\frac{\mu}{\mu-q+2}}}\,dx+1\right),\qquad\mbox{ with } C=C(N,t,q,\mu,h_0).
\end{equation}
Again, we also have
\begin{equation}\label{eq2tsub:estI23p}
\begin{split}
(|\mathcal{I}_{2,q,t,A}|+|\mathcal{I}_{3,q,t,A}|)&\leq  CA\,\left(1+\int_{B_{R}}\left|\frac{\delta_h u}{|h|^{\frac{1+\theta\beta}{\beta}}}\right|^{\frac{\beta\,\mu}{\mu-q+2}}\,dx\right),
\end{split}
\end{equation}
where $C=C(N,s,p,t,q,h_0)>0$.

 By inserting the estimates \eqref{eq2tsub:estI1q}, \eqref{eq2tsub:estI2q}, \eqref{eq2tsub:estI3q}, \eqref{eq2tsub:estI1p}, \eqref{eq2tsub:estI11}, \eqref{eq2tsub:estI12} and \eqref{eq2tsub:estI23p} in \eqref{eq2tsub:Il}, we obtain
\begin{equation}\label{eq2tsub:estIl}
\begin{split}
\left[\frac{|\delta_h u|^\frac{\beta-1}{2}\,\delta_h u}{|h|^\frac{1+\theta\,\beta}{2}}\,\eta^\frac{q}{2}\right]^2_{W^{\sigma,2}(B_R)}&\leq  C(A+1)\,\left(1+\int_{B_{R}}\left|\frac{\delta_h u}{|h|^{\frac{1+\theta\beta}{\beta}}}\right|^{\frac{\beta\,\mu}{\mu-q+2}}\,dx+\sup_{0<|h|< h_0}\left\|\frac{\delta^2_h u}{|h|^t}\right\|_{L^{\mu}(B_{R+4h_0})}^{\mu}\right)
\end{split}
\end{equation}
where $C=C(N,s,p,t,q,\beta,\mu,h_0)>0$. Here we have used that $p\geq 1$, $|u|\leq 1$ in $B_1$ and Young's inequality to estimate the terms coming from \eqref{eq2tsub:estI1q}. Now, we estimate the term in the L.H.S. of \eqref{eq2tsub:estIl} exactly as in \cite[estimate of (3.25), page 5770]{GL} and combine the resulting estimate in \eqref{eq2tsub:estIl} to obtain
\begin{equation*}
\begin{split}
\sup_{0<|h|< h_0}\int_{B_r}\left|\frac{\delta^2_h u}{|h|^\frac{sp+1+\theta\beta}{\beta+1}}\right|^{\beta+1}\,dx&\leq C(A+1)\,\left(1+\int_{B_{R}}\left|\frac{\delta_h u}{|h|^{\frac{1+\theta\beta}{\beta}}}\right|^{\frac{\beta\,\mu}{\mu-q+2}}\,dx+\sup_{0<|h|< h_0}\left\|\frac{\delta^2_h u}{|h|^t}\right\|_{L^{\mu}(B_{R+4h_0})}^{\mu}\right)
\end{split}
\end{equation*}
for $ C=C(N,s,p,t,q,\beta,\mu,h_0)>0 $, where the last inequality above follows by using that $|u|\leq 1$ in $B_1$ and $q\geq 2$. 

Estimating first order differential quotients with second order ones and making the choices
$$
t=\frac{1+\theta \beta}{\beta}, \quad \beta = \mu-q+2
$$ 
and using that $sp+1+\theta\beta = sp+\beta t\geq (\beta+q)t$, since $sp\geq tq$, we obtain
\begin{equation}\label{eq2tsub:estIlfinaltmu}
\begin{split}
\sup_{0<|h|< h_0}\int_{B_r}\left|\frac{\delta^2_h u}{|h|^\frac{t(\mu+2)}{\mu-q+3}}\right|^{\mu-q+3}\,dx&\leq C(A+1)\,\left(1+\sup_{0<|h|< h_0}\left\|\frac{\delta^2_h u}{|h|^t}\right\|_{L^{\mu}(B_{R+4h_0})}^{\mu}\right),
\end{split}
\end{equation}
where $C=C(N,s,p,t,q,\mu,h_0)>0$. Since $|u|\leq 1$ in $B_1$ and $q\geq 2$, we may also conclude
\[
\begin{split}
\sup_{0<|h|< h_0}\int_{B_r}\left|\frac{\delta^2_h u}{|h|^{t}}\right|^{\mu+2} &= \sup_{0<|h|< h_0}\int_{B_r}\frac{|\delta^2_h u|^{\mu+2}}{|h|^{t(\mu+2)}}\\
&\leq  3^{q-1} \sup_{0<|h|< h_0}\int_{B_r}\frac{|\delta^2_h u|^{\mu-q+3}}{|h|^{t(\mu+2)}}\\
&=3^{q-1}\sup_{0<|h|< h_0}\int_{B_r}\left|\frac{\delta^2_h u}{|h|^\frac{t(\mu+2)}{\mu-q+3}}\right|^{\mu-q+3}.
\end{split}
\]
Therefore, using the above estimate in \eqref{eq2tsub:estIlfinaltmu}, we obtain
\begin{equation}\label{eq2tsub:estIlfinalfinal}
\begin{split}
\sup_{0<|h|< h_0}\int_{B_r}\left|\frac{\delta^2_h u}{|h|^t}\right|^{\mu+2}\,dx&\leq C(A+1)\,\left(1+\sup_{0<|h|< h_0}\left\|\frac{\delta^2_h u}{|h|^t}\right\|_{L^{\mu}(B_{R+4h_0})}^{\mu}\right),
\end{split}
\end{equation}
where $C=C(N,s,p,t,q,\mu,h_0)>0$, which is the desired result.
\end{proof}

We will now conclude the almost $t$-regularity with a slight difference in how the estimate depends on the constant $A$ in the two cases $sp\geq tq$ and $sp\leq tq$.  The proof is almost identical with the proof of \cite[Theorem 4.2]{BLS}. We provide the details for completeness. 
\begin{Theorem}[Almost $C^t$-regularity]
\label{teo:almostt}
Let $\Omega\subset\mathbb{R}^N$ be a bounded and open set and $A>0$. Suppose $u\in X(\Omega)$ is a local weak solution of equation \eqref{eq:maineqnew} in $\Omega$. 
Then $u\in C^\delta_{\rm loc}(\Omega)$ for every $0<\delta<t$. 
\par
 In particular, for every $0<\delta<t$ and every $r>0$  such that  $B_{2r}(x_0)\Subset\Omega$, there are constants $d=d(N,t,q,\delta)\geq 1$ and $C=C(N,s,p,t,q,\delta)>0$ such 
\begin{equation*}
\begin{split}
[u]_{C^\delta(B_{r/2}(x_0))}
\leq \frac{C(A^{-1}r^{tq-sp}\mathcal{M}_r^{p-q}+1)^d}{r^\delta}\mathcal{M}_r,
\end{split}
\end{equation*}
where 
$$
\mathcal{M}_r=\|u\|_{L^\infty(B_{r}(x_0))}+\mathrm{Tail}_{p-1,s\,p}(u;x_0,r)+\mathrm{Tail}_{q-1,t\,q}(u;x_0,r)+1.
$$
If in addition, $sp\geq tq$, then we also have the estimate
\begin{equation*}
\begin{split}
[u]_{C^\delta(B_{r/2}(x_0))}
\leq \frac{C(Ar^{sp-tq}\mathcal{M}_r^{q-p}+1)^d}{r^\delta}\mathcal{M}_r.
\end{split}
\end{equation*}
\end{Theorem}

\begin{proof} We only prove the second estimate. The first estimate follows using Proposition \ref{prop:improvesuper} instead of Proposition \ref{prop:improvesub} and noting that the dependencies on $A$ in those propositions are reversed. By \cite[Proposition 3.3]{GKS}, we have $u\in L^\infty_{\mathrm{loc}}(\Omega)$.
 Without loss of generality, we assume $x_0=0$ and set 
\[
\mathcal{M}_r=\|u\|_{L^\infty(B_{r})}+\mathrm{Tail}_{p-1,s\,p}(u;0,r)+\mathrm{Tail}_{q-1,t\,q}(u;0,r)+1>0.
\]
 It is sufficient to prove that the function 
\[
u_r(x):=\frac{1}{\mathcal{M}_r}\,u(r\,x),\qquad \mbox{ for }x\in B_2,
\]
satisfies the estimate
\[
[u_r]_{C^{\delta}(B_{\frac{1}{2}})}\leq C(Ar^{sp-tq}\mathcal{M}_r^{q-p}+1)^d,
\]
with $C=C(N,s,p,t,q,\delta)>0$ and $d=d(N,t,q,\delta)\geq 1$.  Indeed, by scaling, we would recover  the desired estimate. Observe that by definition, the function $u_r$ 
is a local weak solution of $(-\Delta_p)^s u+Ar^{sp-tq}\mathcal{M}_r^{q-p}(-\Delta_q)^t u=0$ in $B_2$ and satisfies
\begin{equation}
\label{eq2t:assumption}
\|u_r\|_{L^\infty(B_1)}\leq 1,\qquad \int_{\mathbb{R}^N\setminus B_1}\frac{|u_r(y)|^{l-1}}{|y|^{N+ml}}\,  dy\leq 1,
\end{equation}
for $(l,m)\in\{(p,s),(q,t)\}$. 

 In order to simplify the presentation we will from now on omit the subscript $r$ and simply write $u$ in place of $u_r$. 
\vskip.2cm\noindent
We fix $0<\delta<t$ and choose $i_\infty\in\mathbb{N}\setminus\{0\}$ such that
\[
t-\delta> \frac{N}{q+2i_\infty}.
\]
 and  define the sequence of exponents  
\[
q_i=q+2i,\qquad i=0,\dots,i_\infty.
\]
 In addition, we define   
$$
h_0=\frac{1}{64\,i_\infty},\qquad R_i=\frac{7}{8}-4\,(2i+1)\,h_0=\frac{7}{8}-\frac{2i+1}{16\,i_\infty},\qquad \mbox{ for } i=0,\dots,i_\infty,
$$
and note that this implies
\[
R_0+4\,h_0=\frac{7}{8}\qquad \mbox{ and }\qquad R_{i_\infty-1}-4\,h_0=\frac{3}{4}.
\] 
By applying Proposition \ref{prop:improvesub} with
\[
R=R_i\qquad \mbox{ and }\qquad \mu=q_i=q+2i,\qquad \mbox{ for } i=0,\ldots,i_\infty-1,
\] 
and observing that $R_i-4\,h_0=R_{i+1}+4\,h_0$, $4h_0<R_i\leq 1-5h_0$, we obtain the  following chain of  inequalities
\[
\left\{\begin{array}{rcll}
\sup\limits_{0<|h|< h_0}\left\|\dfrac{\delta^2_h u}{|h|^{t}}\right\|_{L^{q_1}(B_{R_1+4h_0})}\leq& C (Ar^{sp-tq}\mathcal{M}_r^{q-p}+1)\,\sup\limits_{0<|h|< h_0}\left(\left\|\dfrac{\delta^2_h u }{|h|^t}\right\|_{L^q(B_{\frac{7}{8}})}+1\right)
\\
\sup\limits_{0<|h|< h_0}\left\|\dfrac{\delta^2_h u}{|h|^{t}}\right\|_{L^{q_{i+1}}(B_{R_{i+1}+4h_0})}\leq&C(Ar^{sp-tq}\mathcal{M}_r^{q-p}+1)\,\sup\limits_{0<|h|< h_0}\left(\left\|\dfrac{\delta^2_h u }{|h|^t}\right\|_{L^{q_i}(B_{R_i+4h_0})}+1\right)
\end{array}
\right.
\]
and finally
\begin{align*}
\sup_{0<|h|< {h_0}}\left\|\frac{\delta^2_h u}{|h|^{t}}\right\|_{L^{q_{i_\infty}}(B_{\frac{3}{4}})}&=\sup_{0<|h|< h_0}\left\|\frac{\delta^2_h u}{|h|^{t}}\right\|_{L^{q+i_\infty}(B_{R_{i_\infty-1}-4h_0})}\\
&\leq C(Ar^{sp-tq}\mathcal{M}_r^{q-p}+1)\sup_{0<|h|< h_0}\left(\left\|\frac{\delta^2_h u }{|h|^t}\right\|_{L^{q+i_\infty-1}(B_{R_{i_\infty-1}+4h_0})}+1\right).
\end{align*}
Here $C=C(N,s,p,t,q,\delta,h_0)>0$. Using \eqref{eq2t:assumption} and proceeding exactly as in the proof of \cite[Lemma 3.2]{GKS}, we get
\begin{equation}\label{eq2t:fsn}
[u_r]_{W^{t,q}(B_{\frac{7}{8}+2h_0})}\leq C(Ar^{sp-tq}\mathcal{M}_r^{q-p}+1),
\end{equation}
for some $C=C(N,s,p,t,q)>0$.  By \cite[Proposition 2.6]{Brolin} together with the relation 
$$
\delta_h u =\frac12\left(\delta_{2h} u-\delta^2_{h}u\right),
$$ 
we have
\begin{align*}
\sup_{0<|h|< h_0}\left\|\frac{\delta^2_h u }{|h|^t}\right\|_{L^{q}(B_{\frac{7}{8}})}&\leq C\,\sup_{0<|h|<2\,h_0}\left\|\frac{\delta_h u }{|h|^t}\right\|_{L^{q}(B_{\frac{7}{8}})}\nonumber \\
&\leq C\left([u]_{W^{t,q}(B_{\frac{7}{8}+2\,h_0})}+\|u\|_{L^\infty(B_{\frac{7}{8}+2\,h_0})}\right)\\
&\leq C(Ar^{sp-tq}\mathcal{M}_r^{q-p}+1),\nonumber 
\end{align*}
where we have used \eqref{eq2t:fsn} and the assumptions \eqref{eq2t:assumption} on $u$.  Therefore, the above chain of  inequalities leads us to
\begin{equation}
\sup_{0<|h|< {h_0}}\left\|\frac{\delta^2_h u}{|h|^{t}}\right\|_{L^{q_{i_\infty}}(B_{\frac{3}{4}})}\leq C(Ar^{sp-tq}\mathcal{M}_r^{q-p}+1)^{i_\infty+1},
\label{eq2t:otherest}
\end{equation}
where $C=C(N,s,p,t,q,\delta,h_0)>0$.
Take $\chi\in C_0^\infty(B_{\frac{5}{8}})$ such that 
$$
0\leq \chi\leq 1, \qquad \chi \equiv 1 \text{ in $B_{\frac{1}{2}}$},\qquad |\nabla \chi|\leq C,\qquad |D^2 \chi|\leq C,
$$
for some $C=C(N)>0$.
In particular 
$$
\frac{|\delta_h\chi|}{|h|^t}\leq C,\qquad \frac{|\delta^2_h\chi|}{|h|^t}\leq C,
$$
for all $|h| > 0$, where $C=C(N)>0$. We also note that
$$
\delta^2_h (u\,\chi)=\chi_{2h}\,\delta^2_h u+2\,\delta_h u\, \delta_h \chi_h+u\,\delta^2_h\chi.
$$
Hence, 
\begin{align*}
[u\,\chi]_{\mathcal{B}^{t,q_{i_\infty}}_\infty(\mathbb{R}^N)}&\leq C\left(\sup_{0<|h|< h_0}\left\|\frac{\delta^2_h (u\,\chi)}{|h|^t}\right\|_{L^{q_{i_\infty}}(\mathbb{R}^N)}+1\right)\\&\leq C\sup_{0<|h|< h_0}\,\left(\left\|\frac{\chi_{2h}\,\delta^2_h u}{|h|^t}\right\|_{L^{q_{i_\infty}}(\mathbb{R}^N)}+\left\|\frac{\delta_h u\,\delta_h\chi}{|h|^t}\right\|_{L^{q_{i_\infty}}(\mathbb{R}^N)}+\left\|\frac{u\,\delta^2_h\chi}{|h|^t}\right\|_{L^{q_{i_\infty}}(\mathbb{R}^N)}+1\right) \nonumber\\
&\leq C\sup_{0<|h|< h_0}\,\left(\left\|\frac{\delta^2_h u}{|h|^t}\right\|_{L^{q_{i_\infty}}(B_{\frac{5}{8}+2\,h_0})}+\|\delta_h u\|_{L^{q_{i_\infty}}(B_{\frac{5}{8}+h_0})}+\|u\|_{L^{q_{i_\infty}}(B_{\frac{5}{8}+2h_0})}+1\right) \\
&\leq C\sup_{0<|h|< h_0}\,\left(\left\|\frac{\delta^2_h u}{|h|^t}\right\|_{L^{q_{i_\infty}}(B_{\frac{3}{4}})}+\|u\|_{L^{q_{i_\infty}}(B_{\frac{3}{4}})}+1\right)\\
&\leq C(Ar^{sp-tq}\mathcal{M}_r^{q-p}+1)^{i_\infty+1}, \nonumber
\end{align*}
by \eqref{eq2t:otherest}, where $C=C(N,s,p,t,q,\delta,h_0)>0$. By Lemma \ref{emb1}, we have
\begin{equation*}
[u\,\chi]_{\mathcal{N}_\infty^{t,q_{i_\infty}}(\mathbb{R}^N)}\leq  C(N,\delta,s)(\,[u\,\chi]_{\mathcal{B}_\infty^{t,q_{i_\infty}}(\mathbb{R}^N)}+1) \leq C(Ar^{sp-tq}\mathcal{M}_r^{q-p}+1)^{i_\infty+1},
\end{equation*}
where $C=C(N,s,p,t,q,\delta,h_0)>0$.
Finally,  by the choice  of $i_\infty$ we have
\[
t\,q_{i_\infty}>N\qquad \mbox{and }\qquad \delta<t-\frac{N}{q_{i_\infty}}.
\] 
 We may therefore apply  Theorem~\ref{emb2} with  $\beta=s$, $\alpha=\delta$ and $q=q_{i_\infty}$ to obtain
\begin{align*}
[u]_{C^\delta(B_{\frac{1}{2}})}&= [u\,\chi]_{C^\delta(B_{\frac{1}{2}})}\\
&\leq C\left([u\,\chi]_{\mathcal{N}_\infty^{t,q_{i_\infty}}(\mathbb{R}^N)}\right)^{\frac{\delta\,q_{i_\infty}+N}{t\,q_{i_\infty}}}\,\left(\|u\,\chi\|_{L^q(\mathbb{R}^N)}\right)^\frac{(t-\delta)\,q_{i_\infty}-N}{t\,q_{i_\infty}}\\
&\leq C(Ar^{sp-tq}\mathcal{M}_r^{q-p}+1)^\frac{(\delta q_{i_\infty}+N)(i_\infty+1)}{tq_{i_\infty}}\\
&\leq C(Ar^{sp-tq}\mathcal{M}_r^{q-p}+1)^d,
\end{align*}
where $C=C(N,s,p,t,q,\delta,h_0)>0$ and $d=\max\left(1,\frac{(\delta q_{i_\infty}+N)(i_\infty+1)}{tq_{i_\infty}}\right)\geq 1$.
This concludes the proof upon recalling the dependencies of $h_0$ and $i_\infty$.
\end{proof}

\subsection{Higher H\"older regularity using the subquadratic scheme}
 In this subsection, we improve the H\"older regularity using the subquadratic scheme.
\begin{prop}
\label{prop:improve3} (Improved Besov regularity using the subquadratic scheme)
Suppose $A>0$ and let $u\in X(B_2)$ be a local weak solution of equation \eqref{eq:maineqnew} in $B_2$.
Suppose that 
\begin{equation}
\label{eq2improvesub:boundsimprovebesovsub}
\|u\|_{L^\infty(B_1)}\leq 1,\quad  \int_{\mathbb{R}^N\setminus B_1} \frac{|u(y)|^{l-1}}{|y|^{N+m\,l}}\,  dy\leq 1, \quad \text{and}\quad [u]_{C^\gamma(B_1)}\,  \leq 1,
\end{equation} 
for some $$\gamma\in \left(\max \left(\frac{tq-2}{q-2},0\right),1\right).$$ Here $(l,m)\in \{(s,p),(t,q)\}$.
Then for any $\alpha\in [0,1)$, $q\leq\mu <\infty$, $0<h_0<\frac{1}{10}$, $R$ such that $4\,h_0<R\le 1-5\,h_0$,  
we have
\begin{equation*}
\begin{split}
\sup_{0<|h|< h_0}\left\|\frac{\delta^2_h u}{|h|^{\frac{sp-\gamma(p-2)+\alpha\mu}{\mu+1}}}\right\|_{L^{\mu+1}(B_{R-4\,h_0})}^{{\mu}+1}&\leq C\Big({A}+1\Big)\,\left(\sup_{0<|h|< h_0}\left\|\frac{\delta^2_h u }{|h|^\alpha}\right\|_{L^{\mu}(B_{R+4h_0})}^{\mu}+1\right),
\end{split}
\end{equation*}
for some positive constant $C=C(N,s,p,t,q,\alpha,\gamma,\mu,h_0)$.
\end{prop}

\begin{proof}
The setup and the starting point is the same as in the proof of Proposition \ref{prop:improvesuper}. We also arrive at 
\begin{equation*}
\begin{split}
&\sum_{(l,m,k)\in\{(p,s,1),(q,t,A)\}}k\iint_{\mathbb{R}^N\times\mathbb{R}^N} \frac{\Big(J_l(u_h(x)-u_h(y))-J_l(u(x)-u(y))\Big)}{|h|^{1+\theta\beta}}\\
&\times\Big(J_{\beta+1}(u_h(x)-u(x))\,\eta(x)^q-J_{\beta+1}(u_h(y)-u(y))\,\eta(y)^q\Big)\,d\mu_{l,m}=0,
\end{split}
\end{equation*}
that is 
\begin{equation}
\label{eq2improvesub:Il}
\sum_{(l,m,k)\in\{(p,s,1),(q,t,A)\}}(\mathcal{I}_{1,l,m,k}+\mathcal{I}_{2,l,m,k}+\mathcal{I}_{3,l,m,k})=0,
\end{equation}
where
\[
\begin{split}
\mathcal{I}_{1,l,m,k}:=k\iint_{B_R\times B_R}& \frac{\Big(J_l(u_h(x)-u_h(y))-J_l(u(x)-u(y))\Big)}{|h|^{1+\theta\beta}}\\
&\times\Big(J_{\beta+1}(u_h(x)-u(x))\,\eta(x)^q-J_{\beta+1}(u_h(y)-u(y))\,\eta(y)^q\Big)d\mu_{l,m},
\end{split}
\]
\[
\begin{split}
\mathcal{I}_{2,l,m,k}:=k\iint_{B_\frac{R+r}{2}\times (\mathbb{R}^N\setminus B_R)}& \frac{\Big(J_l(u_h(x)-u_h(y))-J_l(u(x)-u(y))\Big)}{|h|^{1+\theta\,\beta}}\\
&\times J_{\beta+1}(u_h(x)-u(x))\,\eta(x)^q\,d\mu_{l,m},
\end{split}
\]
and
\[
\begin{split}
\mathcal{I}_{3,l,m,k}:=-k\iint_{(\mathbb{R}^N\setminus B_R)\times B_\frac{R+r}{2}}& \frac{\Big(J_l(u_h(x)-u_h(y))-J_l(u(x)-u(y))\Big)}{|h|^{1+\theta\,\beta}}\\
\times &
 J_{\beta+1}(u_h(y)-u(y))\,\eta(y)^q\,d\mu_{l,m},
\end{split}
\]
where we used that $\eta$ vanishes identically outside $B_{\frac{R+r}{2}}$.

We estimate $\mathcal{I}_{j,l,m,k}$ for $j=1,2,3$ and $(l,m,k)\in\{(p,s,1),(q,t,A)\}$ separately, starting with $(l,m)=(p,s)$. Note that $\eta^q=(\eta^\frac{q}{2})^2$ and that $\eta^\frac{q}{2}$ is Lipschitz and enjoys the same properties as $\eta$. Since $p\leq 2$ and \eqref{eq2improvesub:boundsimprovebesovsub} holds, we may proceed along the lines of the proof of the estimates of $\mathcal{I}_1, \mathcal{I}_2$ and $\mathcal{I}_3$ in \cite{GL} (see pages 5767-6768 and estimates (3.17), (3.18)), and obtain for $2\sigma=sp-\gamma(p-2)$ that
\begin{equation}\label{eq2improvesub:estI1q}
\begin{split}
\mathcal{I}_{1,p,s,1}
&\geq {c}\left[\frac{|\delta_h u|^\frac{\beta-1}{2}\,\delta_h u}{|h|^\frac{1+\theta\,\beta}{2}}\,\eta^{\frac{q}{2}}\right]^2_{W^{\sigma,2}(B_R)}-C\int_{B_R} \frac{|\delta_h u(x)|^{p+\beta-1}}{|h|^{1+\theta\,\beta}} dx\\
& -C\,\int_{B_R}\, \frac{|\delta_h u(x)|^{\beta+1}}{|h|^{1+\theta\,\beta}}dx.
\end{split}
\end{equation}
Here $c=c(p,\beta)>0$ and $C=C(N,s,p,\beta,h_0)>0$.

Moreover, 
\begin{equation}
\label{eq2improvesub:estI2q}
|\mathcal{I}_{2,p,s,1}|\leq C\int_{B_R}\frac{|\delta_h u(x)|^{\beta}}{|h|^{1+\theta\,\beta}} dx,
\end{equation}
and
\begin{equation}
\label{eq2improvesub:estI3q}
|\mathcal{I}_{3,p,s,1}|\leq C\int_{B_R}\frac{|\delta_h u(x)|^{\beta}}{|h|^{1+\theta\,\beta}} dx.
\end{equation}
Here $C=C(N,s,p,h_0)>0$ is a constant.

Next, we estimate $\mathcal{I}_{j,l,m,1}$ for $j=1,2,3$ and $(l,m)=(q,t)$. Again, since $q\geq 2$, we may, taking the properties of $\eta$ into account, proceed along the lines of the proof of the estimates of $\mathcal{I}_1, \mathcal{I}_{11}$ and $\mathcal{I}_{12}$ from pages 814-819 in \cite{BLS}, and obtain
\begin{equation}\label{eq2improvesub:estI1p}
\begin{split}
\mathcal{I}_{1,q,t,A}\ge cA& \left[\frac{|\delta_h u|^\frac{\beta-1}{q}\,\delta_h u}{|h|^\frac{1+\theta\,\beta}{q}}\,\eta\right]^q_{W^{t,q}(B_R)}-CA\,\mathcal{I}_{11}-CA\,\mathcal{I}_{12}
\end{split}
\end{equation}
where $c=c(q,\beta)>0$ and $C=C(q,\beta)>0$, where we set 
\begin{equation*}
\begin{split}
\mathcal{I}_{11}:=\,\iint_{B_R\times B_R} &\left(|u_h(x)-u_h(y)|^\frac{q-2}{2}+|u(x)-u(y)|^\frac{q-2}{2}\right)^2\\
&\times \left|\eta(x)^\frac{q}{2}-\eta(y)^\frac{q}{2}\right|^2\, \frac{|\delta_h u(x)|^{\beta+1}+|\delta_h u(y)|^{\beta+1}}{|h|^{1+\theta\,\beta}}\,d\mu_{q,t},
\end{split}
\end{equation*}
and
\[
\begin{split}
\mathcal{I}_{12}&:=\,\iint_{B_R\times B_R}\, \left(\frac{|\delta_h u(x)|^{\beta-1+q}}{|h|^{1+\theta\,\beta}}+\frac{|\delta_h u(y)|^{\beta-1+q}}{|h|^{1+\theta\,\beta}}\right)\, |\eta(x)-\eta(y)|^q\,d\mu_{q,t}.
\end{split}
\]
In contrast to the proof of Proposition \ref{prop:improvesub} above, we estimate the term $\mathcal{I}_{11}$ as follows. Since $[u]_{C^{\gamma}}(B_1)\leq 1$, we deduce that
$$
 \left(|u_h(x)-u_h(y)|^\frac{q-2}{2}+|u(x)-u(y)|^\frac{q-2}{2}\right)^2\,\left|\eta(x)^\frac{q}{2}-\eta(y)^\frac{q}{2}\right|^2|x-y|^{-N-tq}\leq C | x-y|^{-N-tq+2+(q-2)\gamma}
$$
for $x,y\in B_R$ where the exponent is strictly larger than $-N$ by the choice of $\gamma$. This implies
\begin{equation}\label{eq2improvesub:estI11}
|\mathcal{I}_{11}|\leq C\,\int_{B_R}\frac{|\delta_h u(x)|^{\beta}}{|h|^{1+\theta\,\beta}}\,dx
\end{equation}
where  $C=C(N,t,q,h_0)>0$. As in the estimate of $\mathcal{I}_{12}$ on page 828 in \cite{BLS}, we get
\begin{equation}\label{eq2improvesub:estI12}
|\mathcal{I}_{12}|\leq C\,\int_{B_R}\frac{|\delta_h u(x)|^{\beta}}{|h|^{1+\theta\,\beta}}\,dx ,\qquad\mbox{ with } C=C(N,t,q,h_0)>0.
\end{equation}
Following estimate (4.9) in \cite{BLS}, we also obtain
\begin{equation}\label{eq2improvesub:estI23p}
\begin{split}
(|\mathcal{I}_{2,q,t,A}|+|\mathcal{I}_{3,q,t,A}|)&\leq CA\,\int_{B_{\frac{R+r}{2}}}\frac{|\delta_h u|^{\beta}}{|h|^{1+\theta\beta}} dx
\end{split}
\end{equation}
where $C=C(N,t,q,h_0)>0$. By inserting the estimates \eqref{eq2improvesub:estI1q}, \eqref{eq2improvesub:estI2q}, \eqref{eq2improvesub:estI3q}, \eqref{eq2improvesub:estI1p}, \eqref{eq2improvesub:estI11}, \eqref{eq2improvesub:estI12} and \eqref{eq2improvesub:estI23p} in \eqref{eq2improvesub:Il}, we obtain
\begin{equation}\label{eq2improvesub:estIl}
\begin{split}
\left[\frac{|\delta_h u|^\frac{\beta-1}{2}\,\delta_h u}{|h|^\frac{1+\theta\,\beta}{2}}\,\eta^\frac{q}{2}\right]^2_{W^{\sigma,2}(B_R)}&\leq  C(A+1)\int_{B_R}\frac{|\delta_h u(x)|^{\beta}}{|h|^{1+\theta\,\beta}}dx
\end{split}
\end{equation}
where $C=C(N,s,p,t,q,\beta,h_0)>0$. By estimating the term in the L.H.S. of \eqref{eq2improvesub:estIl} exactly as in (3.26) of \cite{GL} and combining this with the resulting estimate in \eqref{eq2improvesub:estIl} we obtain
\begin{equation*}
\begin{split}
\sup_{0<|h|< h_0}\int_{B_r}\left|\frac{\delta^2_h u}{|h|^\frac{2\sigma+1+\theta\beta}{\beta+1}}\right|^{\beta+1}\,dx&\leq C(A+1)\,\int_{B_{R}}\frac{|\delta_h u|^\beta}{|h|^{{1+\theta\beta}}},
\end{split}
\end{equation*}
for $C=C(N,s,p,t,q,\beta,\gamma,h_0)>0$, where the last inequality above follows by using  that $|u|\leq 1$ in $B_1$. 

Now, taking into account the assumption \eqref{eq2improvesub:boundsimprovebesovsub}, the result follows proceeding along the lines of the proof of (3.27) and (3.3) in \cite{GL}.
\end{proof}

We can now iterate the improved Besov-type regularity above to obtain an improved H\"older regularity exactly as in Proposition 3.2 of \cite{GL}. The only differences are that there is a factor $(A+1)$ picked up at each step of the iteration and that we  need to impose a lower bound on $\gamma$. We provide the details for completeness.
\begin{prop}
\label{prop:improve2} (Improved H\"older regularity using the subquadratic scheme)
Let $A>0$ and $u\in X(B_2)$ be a local weak solution of equation \eqref{eq:maineqnew} in $B_2$. Suppose that 
\begin{equation}
\label{eq2improveH:boundsimp}
\|u\|_{L^\infty(B_1)}\leq 1,\quad  \int_{\mathbb{R}^N\setminus B_1} \frac{|u(y)|^{l-1}}{|y|^{N+m\,l}}\,  dy\leq 1, \quad \text{and}\quad [u]_{C^\gamma(B_1)}\,  \leq 1,
\end{equation} 
for some $$\gamma\in \left(\max \left(\frac{tq-2}{q-2},0\right),1\right).$$  Here $(l,m)\in \{(s,p),(t,q)\}$. Define $\tau=\min (sp-\gamma(p-2),1)$. Then for any $\e\in(0,\tau)$, we have
$$
 [u]_{C^{\tau-\e}(B_\frac12)}\leq C(1+A)^d,
$$
where $d=d(N,s,p,\e,\gamma)\geq 1$ and $C=C(N,s,p,t,q,\e,\gamma)>0$ are constants.
\end{prop}
\begin{proof}
 Pick  $0<\e<\tau$ and choose $\lambda$ so that 
$$
 \tau-\frac{\e}{2}-\frac{N}{\lambda}>\tau-\e>0.
$$
Furthermore, define the sequence of exponents  
\[
\alpha_0=0, \quad \alpha_i=\frac{sp-\gamma(p-2)+\alpha_{i-1} \lambda}{\lambda+1},\qquad i=0,\dots,i_\infty,
\]
where {$i_\infty\geq 1$}, depending on $N,s,p,\e$ and $\gamma$,  is chosen  such that 
$$
\alpha_{i_\infty-1}<\tau-\frac{\e}{2}\leq \alpha_{i_\infty}.
$$
 This is indeed possible as $\alpha_i$ are increasing toward $sp-\gamma(p-2)$.  We also introduce 
$$
h_0=\frac{1}{64\,i_\infty},\qquad R_i=\frac{7}{8}-4\,(2i+1)\,h_0=\frac{7}{8}-\frac{2i+1}{16\,i_\infty},\qquad \mbox{ for } i=0,\dots,i_\infty.
$$
 Note that this implies
\[
R_0+4\,h_0=\frac{7}{8}\qquad \mbox{ and }\qquad R_{i_\infty-1}-4\,h_0=\frac{3}{4}.
\] 
Taking into account \eqref{eq2improveH:boundsimp} and applying Proposition \ref{prop:improve3} and with $R=R_i$ and $\mu=\lambda$, and observing that $R_i-4\,h_0=R_{i+1}+4\,h_0$,
we  arrive at the following chain of  inequalities
\[
\left\{\begin{array}{rcll}
\sup\limits_{0<|h|< h_0}\left\|\dfrac{\delta^2_h u}{|h|^{\alpha_1}}\right\|_{L^{\lambda}(B_{R_1+4h_0})}&\leq& C(A+1)\,\sup\limits_{0<|h|< h_0}\left(\left\|\delta^2_h u \right\|_{L^\lambda (B_{\frac{7}{8}})}+1\right)\\
&&&\\
\sup\limits_{0<|h|< h_0}\left\|\dfrac{\delta^2_h u}{|h|^{\alpha_{i+1}}}\right\|_{L^{\lambda}(B_{R_{i+1}+4h_0})}&\leq& C(A+1)\,\sup\limits_{0<|h|< h_0}\left(\left\|\dfrac{\delta^2_h u }{|h|^{\alpha_i}}\right\|_{L^{\lambda}(B_{R_i+4h_0})}+1\right),
\end{array}
\right.
\]
for $i=1,\ldots,i_\infty-2$, and finally
\[\begin{split}
\sup_{0<|h|< {h_0}}\left\|\frac{\delta^2_h u}{|h|^{\alpha_{i_\infty}}}\right\|_{L^{\lambda}(B_{\frac{3}{4}})}&=\sup_{0<|h|< h_0}\left\|\frac{\delta^2_h u}{|h|^{\alpha_{i_\infty}}}\right\|_{L^{\lambda}(B_{R_{i_\infty-1}-4h_0})}\\
&\leq C(A+1)\sup_{0<|h|< h_0}\left(\left\|\frac{\delta^2_h u }{|h|^{\alpha_{i_\infty-1}}}\right\|_{L^{\lambda}(B_{R_{i_\infty-1}+4h_0})}+1\right).
\end{split}
\]
Here $C=C(N,s,p,t,q,\e,\gamma)>0$. By  \eqref{eq2improveH:boundsimp} we also have 
\begin{align*}\label{eq2improveH::1sttofrac}
\sup\limits_{0<|h|< h_0}\left\|\delta^2_h u \right\|_{L^\lambda (B_{\frac{7}{8}})}&\leq {3}\|u\|_{L^{\infty}(B_{1})}\leq {3}.
\end{align*}
Hence,  the above chain of inequalities implies 
\begin{equation}\label{eq2improveH:neqn11}
\sup_{0<|h|< {h_0}}\left\|\frac{\delta^2_h u}{|h|^{\alpha_{i_\infty}}}\right\|_{L^{\lambda}(B_{\frac{3}{4}})}\leq C(A+1)^{i_\infty},
\end{equation}
where $C=C(N,s,p,t,q,\e,\gamma)>0$  and since $\alpha_{i_\infty}\geq \tau-\frac{\e}{2}$, estimate \eqref{eq2improveH:neqn11}  implies
\begin{equation}\label{eq2improveH:neqn12}
\sup_{0<|h|< {h_0}}\left\|\frac{\delta^2_h u}{|h|^{\tau-\frac{\e}{2}}}\right\|_{L^{\lambda}(B_{\frac{3}{4}})}\leq c(A+1)^{i_\infty},
\end{equation}
where $C=C(N,s,p,t,q,\e,\gamma)>0$.
 Let  $\chi\in C_0^\infty(B_{\frac{5}{8}})$ such that
$$
0\leq \chi\leq 1, \qquad |\nabla \chi|\leq C,\qquad |D^2 \chi|\leq C \text{ in }B_\frac{5}{8}\qquad\text{ and }\qquad \chi \equiv 1 \text{ in }B_{\frac{1}{2}},
$$
for some $C=C(N)>0$.
In particular
$$
\frac{|\delta_h\chi|}{|h|}\leq C,\qquad \frac{|\delta^2_h\chi|}{{|h|^2}}\leq C,
$$for all $|h| > 0$, where $C=C(N)>0$.
We also note that
$$
\delta^2_h (u\,\chi)=\chi_{2h}\,\delta^2_h u+2\,\delta_h u\, \delta_h \chi_h+u\,\delta^2_h\chi.
$$
Hence, using the above properties of $\chi$ and \eqref{eq2improveH:neqn12}, we have 
\begin{align*}
[u\,\chi]_{\mathcal{B}^{\tau-\frac{\e}{2},\lambda}_\infty(\mathbb{R}^N)}&\leq C\left(\sup_{0<|h|< {h_0}}\left\|\frac{\delta^2_h (u\,\chi)}{|h|^{\tau-\frac{\e}{2}}}\right\|_{L^{\lambda}(\mathbb{R}^N)}+1\right)\\&\leq C\sup_{0<|h|< {h_0}} \,\left(\left\|\frac{\chi_{2h}\,\delta^2_h u}{|h|^{\tau-\frac{\e}{2}}}\right\|_{L^{\lambda}(\mathbb{R}^N)}+\left\|\frac{\delta_h u\,\delta_h\chi}{|h|^{\tau-\frac{\e}{2}}}\right\|_{L^{\lambda}(\mathbb{R}^N)}+\left\|\frac{u\,\delta^2_h\chi}{|h|^{\tau-\frac{\e}{2}}}\right\|_{L^{\lambda}(\mathbb{R}^N)}+1\right) \nonumber\\
&\leq C\sup_{0<|h|< {h_0}} \,\left(\left\|\frac{\delta^2_h u}{|h|^{\tau-\frac{\e}{2}}}\right\|_{L^{\lambda}(B_{\frac{5}{8}+2\,h_0})}+\|\delta_h u\|_{L^{\lambda}(B_{\frac{5}{8}+h_0})}+\|u\|_{L^{\lambda}(B_{\frac{5}{8}+2h_0})}+1\right) \\
&\leq C\sup_{0<|h|< {h_0}} \,\left(\left\|\frac{\delta^2_h u}{|h|^{\tau-\frac{\e}{2}}}\right\|_{L^{\lambda}(B_{\frac{3}{4}})}+\|u\|_{L^{\lambda}(B_{\frac{3}{4}})}+1\right)\\
&\leq C(A+1)^{i_\infty}, \nonumber
\end{align*}
where $C=C(N,s,p,t,q,\e,\gamma)>0$.
Thus by the above estimate and Lemma \ref{emb1}, we have
\begin{equation*}\label{eq2improveH:Neq2}
[u\,\chi]_{\mathcal{N}_\infty^{\tau-\frac{\e}{2},\lambda}(\mathbb{R}^N)}\leq  C(N,\e,\lambda)\, ([u\,\chi]_{\mathcal{B}_\infty^{\tau-\frac{\e}{2},\lambda}(\mathbb{R}^N)}+1) \leq C(A+1)^{i_\infty},
\end{equation*}
where $C=(N,s,p,t,q,\e,\gamma)>0$.  We note that by the   choice of   $\lambda$  we have
\[
 \tau-\e<\tau-\frac{\e}{2}-\frac{N}{\lambda}.
\] 
 Therefore, we may  apply Theorem \ref{emb2} with  $\beta=\tau-\frac{\e}{2}$, $\alpha=\tau-\e$  and $q=\lambda$  to obtain
\[
\begin{split}
[u]_{C^{\tau-\e}(B_{\frac{1}{2}})}&= [u\,\chi]_{C^{\tau-\e}(B_{\frac{1}{2}})}\\
&\leq C\left([u\,\chi]_{\mathcal{N}_\infty^{\tau-\frac{\e}{2},\lambda}(\mathbb{R}^N)}\right)^{\frac{(\tau-\e)\,\lambda+N}{(\tau-\frac{\e}{2})\,\lambda}}\,\left(\|u\,\chi\|_{L^\lambda (\mathbb{R}^N)}\right)^\frac{\frac{q\e}{2}-N}{(\tau-\frac{\e}{2})\,\lambda}\\
&\leq C(A+1)^{d},
\end{split}
\]
where $C=C(N,s,p,t,q,\e,\gamma)>0$ and $d=\max\left( 1,i_\infty \frac{(\tau-\e)\,\lambda+N}{(\tau-\frac{\e}{2})\,\lambda} \right)\geq 1$.
 This is the desired result.
\end{proof}

\subsubsection{Final H\"older regularity}
 In this subsection, we prove a normalized version of Theorem \ref{teo:1higher} by iterating the improved H\"older regularity obtained in the previous subsection. This is done exactly as in Theorem 3.3 of \cite{GL}.

\begin{Theorem}[Almost $\frac{sp}{p-1}$-regularity]
\label{teo:localalmost} Let $A>0$ and $u\in X(B_4)$ be a local weak solution of equation \eqref{eq:maineqnew} in $B_4$  satisfying
$$
\|u\|_{L^\infty(B_3)}\leq 1,\quad  \int_{\mathbb{R}^N\setminus B_3} \frac{|u(y)|^{l-1}}{|y|^{N+m\,l}}\,  dy\leq 1,
$$
where $(l,m)\in \{(p,s), (t,q)\}$.
Then for any $\e\in(0,\Gamma)$, there is $\sigma(s,p,t,q,\e)>0$ such that $u\in C^{\Gamma-\e}(B_\sigma)$, where
$$
\Gamma = \min\Big(\frac{sp}{p-1},1\Big).
$$
More precisely, there are constants $C=C(N,s,p,t,q,\e)>0$ and  $d=d(N,s,p,t,q,\e) \geq 1$ such that 
$$
[u]_{C^{\Gamma-\e}(B_\sigma)}\leq C(1+(A^{-1})^d+A^d).
$$
If in addition $sp\geq tq$, then we also have the slightly stronger estimate
$$
[u]_{C^{\Gamma-\e}(B_\sigma)}\leq C(1+A^d).
$$
\end{Theorem}
\begin{proof} The idea is to start with Theorem \ref{teo:almostt} and then apply Proposition \ref{prop:improve2} iteratively. More precisely, by Theorem \ref{teo:almostt} and a simple rescaling argument , we have
$$
 [u]_{C^\gamma (B_2)} \leq C(1+A)^d\quad \text{or}\quad  C(1+A^{-1})^d, 
$$
for 
$$
\gamma = \frac12 \left(t+\max\left(\frac{tq-2}{q-2},0\right)\right).
$$
where $C=C(N,s,p,t,q)>0$ and $d=d(N,t,q)\geq 1$, and the first estimate only holds when $sp\geq tq$, while the second one always holds true.
Therefore, the function 
$$
\tilde u(x)=\frac{u(2x)}{\|u\|_{L^\infty(B_{2})}+ \mathrm{Tail}_{p-1,s\,p}(u;0,2)+\mathrm{Tail}_{q-1,t\,q}(u;0,2)+[u]_{C^\frac{tq-2}{q-2}(B_{2})}}
$$
satisfies the assumptions of Proposition \ref{prop:improve2} with
$$
(-\Delta_p)^s \tilde u+\tilde{A} (-\Delta_q)^t \tilde u=0 \text{ in } B_2, 
$$
where
$$
\tilde A=2^{sp-tq}(\|u\|_{L^\infty(B_{2})}+ \mathrm{Tail}_{p-1,s\,p}(u;0,2)+\mathrm{Tail}_{q-1,t\,q}(u;0,2)+[u]_{C^\frac{tq-2}{q-2}(B_{2})})^{q-p}A\leq CK^{q-p}A,
$$
where 
$$
K=C(1+A)^d\quad \text{or}\quad  C(1+A^{-1})^d, \quad C=C(N,s,p,t,q)>0.
$$
Now fix $\e\in(0,\Gamma)$ such that $\e<\frac{2t}{q-1}$\footnote{ We remark that it is enough to prove the result for small $\e$.} and define
$$
\gamma_0=\frac12 \left(t+\max\left(\frac{tq-2}{q-2},0\right)\right),\qquad \gamma_{i+1}=sp-\gamma_i(p-2)-\frac{\e(p-1)}{2}.
$$
Then $\gamma_0<\frac{sp}{p-1}-\frac{\e}{2}$  and  $\gamma_i$ is an increasing sequence and $\gamma_i\to \frac{sp}{p-1}-\frac{\e}{2}$, as $i\to\infty$.  Define also
$$
v_i(x)=\tilde u(2^{-i} x), \quad w_i=\frac{v_i}{M_i}
$$
where
\[\begin{split}
M_i &=  1+ \|\tilde u\|_{L^\infty(B_{2^{-i}})}+ \mathrm{Tail}_{p-1,s\,p}(\tilde u;0,2^{-i})+\mathrm{Tail}_{q-1,t\,q}(\tilde u;0,2^{-i})+2^{-i\gamma_i}[\tilde u]_{C^{\gamma_i}(B_{2^{-i}})}\\
&=1+\|v_i\|_{L^\infty(B_{1})}+ \mathrm{Tail}_{p-1,s\,p}(v_i;0,1)+\mathrm{Tail}_{q-1,t\,q}(v_i;0,1)+[v_i]_{C^{\gamma_i}(B_{1})}\\
&\leq  C(i)(1+[v_i]_{C^{\gamma_i}(B_{1})}).
\end{split}
\]
Note that $w_i$ satisfies
$$
 (-\Delta_p)^s w_i+2^{-i(sp-tq)}M_i^{q-p}\tilde A(-\Delta_q)^t w_i=0\text{ in }B_2.
$$
 We may now find  $i_\infty=i_\infty(s,p,t,q,\e)\in\mathbb{N}$ so that $\gamma_{i_\infty}\geq \Gamma-\e$ and $\gamma_{i_\infty-1}<\Gamma-\e$. Clearly, $M_i\leq C(s,p,t,q,\e)(1+[v_i]_{C^{\gamma_i}(B_{1})}))$ for $0\leq i\leq i_\infty$. Now we apply Proposition \ref{prop:improve2} to $w_i$ successively with $\gamma=\gamma_i$ and $\e$ replaced by $\frac{\e(p-1)}{2}$ and obtain\footnote{Note that in the case $\Gamma=1$, we have $\gamma_{i_\infty-1}<1-\e$ so that Proposition \ref{prop:improve2} applied with $\gamma=\gamma_{i_\infty-2}$ yields the exponent 
$$
\min(\gamma_{i_\infty-1},1-\e (p-1)/2)=\gamma_{i_\infty-1},
$$
since $\gamma_{i_\infty-1}<1-\e$. Therefore, the iteration scheme is intact up this step.}
\begin{equation}
2^{\gamma_1}[v_1]_{C^{\gamma_1}(B_1)}=[v_0]_{C^{\gamma_1}(B_\frac12)}\leq C(1+\tilde A)^{d_0}\leq C(1+\tilde A^{d_0})\label{eq:est0}
\end{equation}
and
\begin{equation}\label{}
\begin{split}
2^{\gamma_{i+1}}[v_{i+1}]_{C^{\gamma_{i+1}}(B_1)}&=M_i[w_i]_{C^{\gamma_{i+1}}(B_\frac12)}\leq M_i C(1+\tilde AM_i^{q-p})^{d_i} 
\leq M_iC(1+\tilde A^{d_i}M_i^{(q-p){d_i}}) \nonumber \\
&=  C(M_i+\tilde A^{d_i} M_i^{(q-p){d_i}+1})\leq C (1+\tilde A^{d_i}) M_i^{(q-p){d_i}+1}\\
&\leq C(1+\tilde A^{d_i})(1+[v_i]_{C^{\gamma_i}(B_1)}^{(q-p){d_i}+1})   \nonumber \\
\end{split}
\end{equation}
for $i=1,\ldots i_{\infty-1}$, and finally
\[\begin{split}
 2^{\Gamma-\e}[v_{i_{\infty}}]_{C^{\Gamma-\e}(B_\frac12)}&\leq [v_{i_{\infty}-1}]_{C^{\min(\gamma_{i_{\infty}},1-\frac{\e(p-1)}{2})}(B_\frac12)}=M_{i_{\infty}-1}[w_{i_{\infty}-1}]_{C^{\min(\gamma_{i_{\infty}},1-\frac{\e(p-1)}{2})}(B_\frac12)}  \nonumber \\
 &\leq M_{i_{\infty}-1}C(1+\tilde AM_{i_{\infty}-1}^{q-p})^{d_{i_\infty-1}} \leq C(1+\tilde A^{d_{i_\infty-1}})(1+[v_{i_\infty-1}]_{C^{\gamma_{i_\infty-1}}(B_1)}^{(q-p){d_{i_\infty-1}}+1})  \nonumber ,
\end{split}
\]
where $d_i$ is the exponent appearing when applying Proposition \ref{prop:improve2}, which may depend on $\gamma_i$. In the above chain of inequalities, all constants depend on $N,s,p,t,q$ and $\e$. This have been omitted in order to keep the readability. Here we have used that $M_i\geq 1$, $d_i\geq 1$ and the convexity of the involved power functions. We note that in the scheme above, if we suppose that 
\begin{equation}
\label{eq:indhypo}
[v_i]_{C^{\gamma_i}(B_1)}\leq C(1+\tilde A^\alpha), \quad \alpha\geq  \max_j d_{j}
\end{equation}
then it follows that 
\begin{equation}
\label{eq:schemealpha}
\begin{split}
[v_{i+1}]_{C^{\gamma_{i+1}}(B_1)} &\leq C(1+\tilde A^{d_i})(1+[v_i]_{C^{\gamma_i}(B_1)}^{(q-p){d_i}+1})\\
&\leq C(1+\tilde A^{d_i})(1+(1+\tilde A^\alpha)^{(q-p){d_i}+1})\\
&\leq C(1+\tilde A^{d_i}+(1+\tilde A^\alpha)^{(q-p){d_i}+1})(1+\tilde A^{d_i})\\
&\leq C(1+\tilde A^{\alpha}+(1+\tilde A^\alpha)^{(q-p){d_i}+1})(1+\tilde A^\alpha)\\
&\leq C(1+(1+\tilde A^\alpha)^{(q-p){d_i}+2})\\
&\leq C(1+\tilde A^{\alpha((q-p){d_i}+2)})\\
&\leq C (1+\tilde A^{\alpha((q-p){\alpha}+2)}).
\end{split}
\end{equation}
Again, we have used convexity and the fact that if $\tilde A\leq 1$ then this is trivially true and otherwise $\tilde A^x$ is increasing in $x$. From estimate \eqref{eq:est0} we can assure that \eqref{eq:indhypo} holds true for $i=1$. Hence, by induction over \eqref{eq:schemealpha} we obtain
$$
[\tilde u]_{C^{\Gamma-\e}(B_{2^{-i_\infty-1}})}=2^{i_\infty(\Gamma-\e)}[v_{i_{\infty}}]_{C^{\Gamma-\e}(B_\frac12)}\leq C(1+\tilde A^\beta), 
$$
where $2\leq \beta = \beta(N,s,p,t,q,\e)$ and $C=C(N,s,p,t,q,\e)>0$. After scaling back to $u$ and using that $\tilde A\leq CK^{q-p}A$ with 
$$
K=C(1+A)^d\quad \text{or}\quad  C(1+A^{-1})^d
$$
this implies the desired result with $\sigma = 2^{-i_\infty}$, noting that $i_\infty$ depends on $s,p,t,q$ and $\e$.
\end{proof}

\subsection{Proof of the main result: Theorem \ref{teo:1higher}}
By a simple rescaling we recover the full statement of the main theorem of this part of the paper. As in the previous results of this section, we will have two alternatives on how the dependence of $A$ will appear.

\begin{proof} By \cite[Proposition 3.3]{GKS}, we have $u\in L^\infty_{\mathrm{loc}}(\Omega)$. We assume that $x_0=0$. The proof for $x_0\neq 0$ follows similarly. Let  $$\mathcal{M}_r=\|u\|_{L^\infty(B_r)}+\mathrm{Tail}_{p-1,sp}(u;0,r)+\mathrm{Tail}_{q-1,tq}(u;0,r)+1$$ and define
$$
w(x)=\frac{u(rx)}{\mathcal{M}_r}.
$$
Then $w$ satisfies
$$
(-\Delta_p)^s w+r^{sp-tq}\mathcal{M}_r^{q-p}(-\Delta_q)^t w=0\text{ in }B_2.
$$
Theorem \ref{teo:localalmost} implies in the case $sp\geq tq$ directly that
$$
[w]_{C^{\Gamma-\e}(B_\sigma)}\leq C\left(1+(r^{sp-tq}\mathcal{M}_r^{q-p})^d\right)\leq C\left(1+r^{sp-tq}\mathcal{M}_r^{q-p}\right)^d,
$$
where 
$C=C(N,s,p,t,q,\e)>0$, $d=d(N,s,p,t,q,\e) \geq 1$ and $\sigma=\sigma(s,p,t,q,\e)$. When {$sp<tq$}, it is implied as well since $0<r\leq 1\leq \mathcal{M}_r$ and $q\geq p$ implies 
$$r^{tq-sp}\mathcal{M}_r^{p-q}\leq 1.
$$
Going back to $u$,  we obtain  the desired result.
\end{proof}

\section{ Case III: $p,q\in(1,2]$}\label{sec:III}
In this section, we assume that $p,q\in(1,2]$ and $s,t\in(0,1)$ satisfy $\frac{sp}{p-1}\geq \frac{tq}{q-1}$ unless stated otherwise.
\subsection{Improved Besov and improved H\"older regularity}
\begin{prop}
\label{prop:improveIII}(Improved Besov regularity)
Assume that $A>0$ and 
$u\in X(B_2)$ is a local weak solution of equation \eqref{eq:maineqnew} in $B_2$  satsifying 
\begin{equation}
\label{bounds1}
\|u\|_{L^\infty(B_1)}\leq 1, \qquad \int_{\mathbb{R}^N\setminus B_1} \frac{|u(y)|^{l-1}}{|y|^{N+lm}}\,dy\leq 1 \qquad \mbox{ and }\quad  [u]_{C^\gamma(B_1)}\,  \leq 1,
\end{equation}
for some $\gamma\in [0,1)$, where $(l,m)\in\{(p,s),(q,t)\}$. Then for $\alpha\in [0,1)$, 
$1\leq\mu <\infty$, $0<h_0<\frac{1}{10}$ and $R$ such that $4\,h_0<R\le 1-5\,h_0$,  
we have
\begin{equation*}
\begin{split}
\sup_{0<|h|< h_0}\left\|\frac{\delta^2_h u}{|h|^{\frac{sp-\gamma(p-2)+\alpha\mu}{\mu+1}}}\right\|_{L^{\mu+1}(B_{R-4\,h_0})}^{{\mu}+1}
&\leq C\Big({A}+1\Big)\,\left(\sup_{0<|h|< h_0}\left\|\frac{\delta^2_h u }{|h|^\alpha}\right\|_{L^{\mu}(B_{R+4h_0})}^{\mu}+1\right),
\end{split}
\end{equation*}
for some positive constant $C=C(N,s,p,t,q,\mu,\gamma,h_0)$.
\end{prop}

\begin{proof} 
Define 
$
r=R-4h_0$ and recall $ d\mu_{l,m}=\frac{dx dy}{|x-y|^{N+lm}}.$ Take $\varphi\in W^{s,p}(B_R)\cap W^{t,q}(B_R)$ vanishing outside $B_{\frac{R+r}{2}}$ and  $0<|h|<h_0$. Testing \eqref{wksol} with $\varphi$ and $\varphi_{-h}$ and performing a change of variable yields 
\begin{equation}
\label{differentiated}
\frac1h \sum_{(l,m,k)\in\{(p,s,1),(q,t,A)\}}\iint_{\mathbb{R}^N\times\mathbb{R}^N} k\Big(J_l(u_h(x)-u_h(y))-J_l(u(x)-u(y))\Big)\,\Big(\varphi(x)-\varphi(y)\Big)\,d\mu_{l,m}=0.
\end{equation}
 Take $\eta\in C^\infty_0(B_R)$ is such that
\[
0\le \eta\le 1\text{ in }B_R,\qquad \eta\equiv 1 \text{ in }B_r,\qquad \eta\equiv 0 \mbox{ in } \mathbb{R}^N\setminus B_{\frac{R+r}{2}},\qquad |\nabla \eta|\le \frac{C}{R-r}=\frac{C}{4\,h_0},
\]
for some constant $C=C(N)>0$.
 Inserting 
\[
\varphi=J_{\mu+1}\left(\frac{\delta_h u}{|h|^{\theta}}\right)\,\eta^2,
\]
 in \eqref{differentiated} yields 
\begin{equation*}
\begin{split}
&\sum_{(l,m,k)\in\{(p,s,1),(q,t,A)\}}\iint_{\mathbb{R}^N\times\mathbb{R}^N} k\frac{\Big(J_l(u_h(x)-u_h(y))-J_l(u(x)-u(y))\Big)}{|h|^{1+\theta\,\mu}}\\
&\times\Big(J_{\mu+1}(u_h(x)-u(x))\,\eta(x)^2-J_{\mu+1}(u_h(y)-u(y))\,\eta(y)^2\Big)\,d\mu_{l,m}=0,
\end{split}
\end{equation*}
that is
\begin{equation}
\label{IIIlmk}
\sum_{(l,m,k)\in\{(p,s,1),(q,t,A)\}}(\mathcal{I}_{1,l,m,k}+\mathcal{I}_{2,l,m,k}+\mathcal{I}_{3,l,m,k})=0,
\end{equation}
where
\[
\begin{split}
\mathcal{I}_{1,l,m,k}:=\iint_{B_R\times B_R}& k\frac{\Big(J_l(u_h(x)-u_h(y))-J_l(u(x)-u(y))\Big)}{|h|^{1+\theta\,\mu}}\\
&\times\Big(J_{\mu+1}(u_h(x)-u(x))\,\eta(x)^2-J_{\mu+1}(u_h(y)-u(y))\,\eta(y)^2\Big)d\mu_{l,m},
\end{split}
\]
\[
\begin{split}
\mathcal{I}_{2,l,m,k}:=\iint_{B_\frac{R+r}{2}\times (\mathbb{R}^N\setminus B_R)}& k\frac{\Big(J_l(u_h(x)-u_h(y))-J_l(u(x)-u(y))\Big)}{|h|^{1+\theta\,\mu}}\\
&\times J_{\mu+1}(u_h(x)-u(x))\,\eta(x)^2\,d\mu_{l,m},
\end{split}
\]
and
\[
\begin{split}
\mathcal{I}_{3,l,m,k}:=-\iint_{(\mathbb{R}^N\setminus B_R)\times B_\frac{R+r}{2}}& k\frac{\Big(J_l(u_h(x)-u_h(y))-J_l(u(x)-u(y))\Big)}{|h|^{1+\theta\,\mu}}\\
&\times J_{\mu+1}(u_h(y)-u(y))\,\eta(y)^2\,d\mu_{l,m},
\end{split}
\]
where we have used that $\eta$ vanishes identically outside $B_{\frac{R+r}{2}}$. Let $(l,m,k)\in \{(p,s,1),(q,t,A)\}$. Then taking into account \eqref{bounds1} and following \cite[Step 1, Proposition 3.1]{GL}, we obtain the following lower bound for $\mathcal{I}_{1,l,m}$, with $2\sigma_{l,m} = lm-\gamma(l-2)$
\begin{equation}\label{III1lmk}
\begin{split}
\mathcal{I}_{1,l,m,k}&\geq {ck}\left[\frac{|\delta_h u|^\frac{\mu-1}{2}\,\delta_h u}{|h|^\frac{1+\theta\,\mu}{2}}\,\eta\right]^2_{W^{\sigma_{l,m},2}(B_R)}-Ck\int_{B_R} \frac{|\delta_h u(x)|^{l+\mu-1}}{|h|^{1+\theta\,\mu}} dx\\
& -Ck\,\int_{B_R}\, \frac{|\delta_h u(x)|^{\mu+1}}{|h|^{1+\theta\,\mu}}dx,
\end{split}
\end{equation}
where $c=c(l,\mu)>0$ and $C=C(N,l,m,\mu,h_0)>0$. Moreover, following \cite[Step 2, Proposition 3.1]{GL}, we obtain 
\begin{equation}
\label{III2lmk}
|\mathcal{I}_{2,l,m,k}|\leq Ck\int_{B_R}\frac{|\delta_h u(x)|^{\mu}}{|h|^{1+\theta\,\mu}} dx.
\end{equation}
{Similarly, we get
\begin{equation}
\label{III3lmk}
|\mathcal{I}_{3,l,m,k}|\leq Ck\int_{B_R}\frac{|\delta_h u(x)|^{\mu}}{|h|^{1+\theta\,\mu}} dx,
\end{equation}
where $C=C(N,l,m,h_0)>0$ is some constant.} Taking this into account and choosing $1+\theta\mu=\alpha$, we may combine the above estimates of \eqref{III1lmk}, \eqref{III2lmk} and \eqref{III3lmk} in \eqref{IIIlmk} and proceed along the lines of \cite[Step 3, Proposition 3.1]{GL} to obtain the desired result.
\end{proof}
We can now iterate the improved Besov-type regularity to obtain an improved H\"older regularity.  The proof is almost identical withe one of Proposition 3.2 of \cite{GL}. We provide the details for completeness.
\begin{prop}
\label{prop:improve2III}(Improved H\"older regularity)
Let $A>0$ and $u\in X(B_2)$ be a local weak solution of equation \eqref{eq:maineqnew} in $B_2$.
Suppose that 
\begin{equation}
\label{boundsimp}
\|u\|_{L^\infty(B_1)}\leq 1, \qquad \int_{\mathbb{R}^N\setminus B_1} \frac{|u(y)|^{l-1}}{|y|^{N+l\,m}}\,dy\leq 1\quad  \mbox{ and }\quad [u]_{C^\gamma(B_1)}\leq 1,
\end{equation} 
where $(l,m)\in\{(p,s), (q,t)\}$ and $\gamma\in[0,1)$. Let $\tau =\min (sp-\gamma(p-2),1)$. Then for any $\e\in (0,\tau)$, there exist constants $d=d(N,s,p,\e,\gamma)\geq 1$ and $C=C(N,s,p,t,q,\e,\gamma)>0$ such that
$$
[u]_{C^{\tau-\e}(B_\frac12)}\leq C(A+1)^d.
$$
\end{prop}
\begin{proof}
 Pick  $0<\e<\tau$. We choose $\mu$ such that $\mu>\frac{2N}{\e}$, that is
\begin{equation}\label{mu}
 \tau-\frac{\e}{2}-\frac{N}{\mu}>\tau-\e>0.
\end{equation}
We define the sequence of exponents $\{\alpha_i\}_{i=0}^{i_\infty}$ as follows: 
\[
\alpha_0=0, \quad \alpha_i=\frac{sp-\gamma(p-2)+\alpha_{i-1} \mu}{\mu+1},\qquad i=0,\dots,i_\infty,
\]
where $i_\infty\geq 1$,  depending on $N,s,p,\gamma$ and $\e$,  is chosen such that 
\begin{equation}\label{mu1}
\alpha_{i_\infty-1}<\tau-\frac{\e}{2}\leq \alpha_{i_\infty}.
\end{equation}
 This is indeed possible as  $\alpha_i$ are increasing towards $sp-\gamma(p-2)$. We also  introduce 
$$
h_0=\frac{1}{64\,i_\infty},\qquad R_i=\frac{7}{8}-4\,(2i+1)\,h_0=\frac{7}{8}-\frac{2i+1}{16\,i_\infty},\qquad \mbox{ for } i=0,\dots,i_\infty.
$$
 Note that this implies 
\[
R_0+4\,h_0=\frac{7}{8}\qquad \mbox{ and }\qquad R_{i_\infty-1}-4\,h_0=\frac{3}{4}.
\] 
Taking into account \eqref{boundsimp}, we apply Proposition \ref{prop:improveIII} and with $R=R_i$ and observing that $R_i-4\,h_0=R_{i+1}+4\,h_0$,
we  arrive at the following chain of  inequalities
\[
\left\{\begin{array}{rcll}
\sup\limits_{0<|h|< h_0}\left\|\dfrac{\delta^2_h u}{|h|^{\alpha_1}}\right\|_{L^{\mu}(B_{R_1+4h_0})}&\leq& C(A+1)\,\sup\limits_{0<|h|< h_0}\left(\left\|\delta^2_h u \right\|_{L^\mu(B_{\frac{7}{8}})}+1\right)\\
&&&\\
\sup\limits_{0<|h|\leq h_0}\left\|\dfrac{\delta^2_h u}{|h|^{\alpha_{i+1}}}\right\|_{L^{\mu}(B_{R_{i+1}+4h_0})}&\leq& C(A+1)\,\sup\limits_{0<|h|< h_0}\left(\left\|\dfrac{\delta^2_h u }{|h|^{\alpha_i}}\right\|_{L^{\mu}(B_{R_i+4h_0})}+1\right),
\end{array}
\right.
\]
for $i=1,\ldots,i_\infty-2$, and finally
\begin{align*}
\sup_{0<|h|< {h_0}}\left\|\frac{\delta^2_h u}{|h|^{\alpha_{i_\infty}}}\right\|_{L^{\mu}(B_{\frac{3}{4}})}&=\sup_{0<|h|< h_0}\left\|\frac{\delta^2_h u}{|h|^{\alpha_{i_\infty}}}\right\|_{L^{\mu}(B_{R_{i_\infty-1}-4h_0})}\\
&\leq C(A+1)\sup_{0<|h|< h_0}\left(\left\|\frac{\delta^2_h u }{|h|^{\alpha_{i_\infty-1}}}\right\|_{L^{\mu}(B_{R_{i_\infty-1}+4h_0})}+1\right).
\end{align*}
Here $C=C(N,s,t,p,q,\e,\gamma)>0$. In addition,
\begin{align*}\label{eq:1sttofrac}
\sup\limits_{0<|h|< h_0}\left\|\delta^2_h u \right\|_{L^{\mu}(B_{\frac{7}{8}})}&\leq {3}\|u\|_{L^{\infty}(B_{1})}\leq {3}.
\end{align*}
Hence, the  above chain of inequalities implies 
\begin{equation}\label{neqn11}
\sup_{0<|h|< {h_0}}\left\|\frac{\delta^2_h u}{|h|^{\alpha_{i_\infty}}}\right\|_{L^{q}(B_{\frac{3}{4}})}\leq C(A+1)^{i_\infty},
\end{equation}
where $C=C(N,s,t,p,q,\e,\gamma)>0$.  Since $\alpha_{i_\infty}\geq \tau-\frac{\e}{2}$ by \eqref{mu1}, estimate \eqref{neqn11} implies  
\begin{equation}\label{neqn12}
\sup_{0<|h|< {h_0}}\left\|\frac{\delta^2_h u}{|h|^{\tau-\frac{\e}{2}}}\right\|_{L^{q}(B_{\frac{3}{4}})}\leq C(A+1)^{i_\infty},
\end{equation}
where $C=C(N,s,t,p,q,\e,\gamma)>0$.
 Let  $\chi\in C_0^\infty(B_{\frac{5}{8}})$ such that
$$
0\leq \chi\leq 1\text{ in }B_\frac{5}{8}, \qquad |\nabla \chi|\leq C,\qquad |D^2 \chi|\leq C \text{ in }B_\frac{5}{8}\qquad\text{ and }\qquad \chi \equiv 1 \text{ in }B_{\frac{1}{2}},
$$
for some constant $C=C(N)>0$.  In particular, for every $|h| > 0$, we have
$$
\frac{|\delta_h\chi|}{|h|}\leq C,\qquad \frac{|\delta^2_h\chi|}{{|h|^2}}\leq C,
$$
for some constant $C=C(N)>0$.  We also note that 
$$
\delta^2_h (u\,\chi)=\chi_{2h}\,\delta^2_h u+2\,\delta_h u\, \delta_h \chi_h+u\,\delta^2_h\chi.
$$
Hence,  the properties of $\chi$ combined with \eqref{neqn12} imply 
\begin{equation}\label{Neq}
\begin{split}
[u\,\chi]_{\mathcal{B}^{\tau-\frac{\e}{2},\mu}_\infty(\mathbb{R}^N)}&\leq C\left(\sup_{0<|h|< {h_0}} \left\|\frac{\delta^2_h (u\,\chi)}{|h|^{\tau-\frac{\e}{2}}}\right\|_{L^{\mu}(\mathbb{R}^N)}+1\right)\\&\leq C\sup_{0<|h|< {h_0}} \,\left(\left\|\frac{\chi_{2h}\,\delta^2_h u}{|h|^{\tau-\frac{\e}{2}}}\right\|_{L^{\mu}(\mathbb{R}^N)}+\left\|\frac{\delta_h u\,\delta_h\chi}{|h|^{\tau-\frac{\e}{2}}}\right\|_{L^{\mu}(\mathbb{R}^N)}+\left\|\frac{u\,\delta^2_h\chi}{|h|^{\tau-\frac{\e}{2}}}\right\|_{L^{\mu}(\mathbb{R}^N)}+1\right)\\
&\leq C\sup_{0<|h|< {h_0}} \,\left(\left\|\frac{\delta^2_h u}{|h|^{\tau-\frac{\e}{2}}}\right\|_{L^{\mu}(B_{\frac{5}{8}+2\,h_0})}+\|\delta_h u\|_{L^{\mu}(B_{\frac{5}{8}+h_0})}+\|u\|_{L^{\mu}(B_{\frac{5}{8}+2h_0})}+1\right) \\
&\leq C\sup_{0<|h|< {h_0}} \,\left(\left\|\frac{\delta^2_h u}{|h|^{\tau-\frac{\e}{2}}}\right\|_{L^{\mu}(B_{\frac{3}{4}})}+\|u\|_{L^{\mu}(B_{\frac{3}{4}})}+1\right)\leq C(A+1)^{i_\infty},
\end{split}
\end{equation}
where $C=C(N,s,t,p,q,\e,\gamma)>0$.
Thus by \eqref{Neq} and Lemma \ref{emb1}, we have
\begin{equation*}\label{Neq2}
[u\,\chi]_{\mathcal{N}_\infty^{\tau-\frac{\e}{2},\mu}(\mathbb{R}^N)}\leq C(N,\e)(\,[u\,\chi]_{\mathcal{B}_\infty^{\tau-\frac{\e}{2},\mu}(\mathbb{R}^N)}+1)\leq C(A+1)^{i_\infty},
\end{equation*}
where $C=(N,s,t,p,q,\e,\gamma)>0$.
 We note that by the choice of $\mu$ in \eqref{mu}, we may apply Theorem \ref{emb2} with $\beta=\tau-\frac{\e}{2}$, $\alpha=\tau-\e$ and $q=\mu$ to obtain  
\begin{equation*}
    \begin{split}
[u]_{C^{\tau-\e}(B_{\frac{1}{2}})}&= [u\,\chi]_{C^{\tau-\e}(B_{\frac{1}{2}})}\\
&\leq C\left([u\,\chi]_{\mathcal{N}_\infty^{\tau-\frac{\e}{2},\mu}(\mathbb{R}^N)}\right)^{\frac{(\tau-\e)\,\mu+N}{(\tau-\frac{\e}{2})\,\mu}}\,\left(\|u\,\chi\|_{L^{\mu}(\mathbb{R}^N)}\right)^\frac{\frac{\mu\e}{2}-N}{(\tau-\frac{\e}{2})\,\mu}\\
&\leq C(A+1)^\frac{i_\infty((\tau-\e)\,\mu+N)}{(\tau-\frac{\e}{2})\,\mu},
\end{split}
\end{equation*}
where $C=C(N,s,t,p,q,\e,\gamma)>0$.
By defining $d=\frac{i_\infty((\tau-\e)\,\mu+N)}{(\tau-\frac{\e}{2})\,\mu}+1$,  this  concludes the proof.
\end{proof}
\subsection{Final H\"older regularity}
We now prove a normalized version of Theorem \ref{teo:1higher} by iterating the improved H\"older regularity  obtained  in Proposition \ref{prop:improve2III}.  The proof closely follows the proof \cite[Theorem 3.3]{GL}. Since some care is taken to control the invovled constants, we provide the details.  
\begin{Theorem}[Almost $\frac{sp}{p-1}$-regularity]
\label{teo:localalmostIII}
Suppose $A>0$ and $u\in X(B_2)$ is a local weak solution of equation \eqref{eq:maineqnew} in $B_2$
satisfying
\begin{equation}\label{newbound}
\|u\|_{L^\infty(B_1)}\leq 1\qquad \mbox{ and }\qquad \int_{\mathbb{R}^N\setminus B_1} \frac{|u(y)|^{l-1}}{|y|^{N+l\,m}}\,  dy\leq 1,
\end{equation}
where $(l,m)\in\{(p,s),(q,t)\}$.
Then for any $\e\in(0,\Gamma)$, there  is a  constant $\sigma=\sigma(\e,s,p)>0$ such that $u\in C^{\Gamma-\e}(B_\sigma)$.
 In particular, there are  constants $\beta=\beta(N,s,p,\e)\geq  1$ and $C=C(N,s,p,t,q,\e)>0$ such that
$$
[u]_{C^{\Gamma-\e}(B_\sigma)}\leq C(1+A^{\beta}).
$$
\end{Theorem}
\par
\begin{proof}  The main point of the proof is to apply Proposition \ref{prop:improve2III} repeatedly . Let $\e\in(0,\Gamma)$ and  define 
$$
\gamma_0=0,\qquad \gamma_{i+1}=sp-\gamma_i(p-2)-\frac{\e(p-1)}{2},
$$
for $i=1,\ldots,j_\infty$, where $j_\infty= j_\infty(s,p,\e)\in\mathbb{N}$ is chosen such that $\gamma_{j_\infty}\geq \Gamma-\e$ and $\gamma_{j_\infty-1}<\Gamma-\e$. Then $\gamma_i$ is an increasing sequence and $\gamma_i\to \frac{sp}{p-1}-\frac{{\e}}{2}$, as $i\to\infty$. For $i=1,\ldots,j_\infty$, we also define
$
v_i(x)=u(2^{-i} x)
$
and
\[\begin{split}
M_i &= 1+\|u\|_{L^\infty(B_{2^{-i}})}+ \mathrm{Tail}_{p-1,s\,p}(u;0,2^{-i})+\mathrm{Tail}_{q-1,t\,q}(u;0,2^{-i})+2^{-i\gamma_i}[u]_{C^{\gamma_i}(B_{2^{-i}})}\\
&\leq C(i)(1+ [v_i]_{C^{\gamma_i}(B_{2^{-i}})}).
\end{split}
\]
Note that $w_i=v_i/M_i$ satisfies
$$
 (-\Delta_p)^s w_i+2^{-i(sp-tq)}M_i^{q-p} A(-\Delta_q)^t w_i=0\text{ in }B_2.
$$
\textbf{Case-I: When $1<p\leq q\leq 2$:}\\
Taking into account \eqref{newbound}, we apply Proposition \ref{prop:improve2III} to $w_i=v_i/M_i$ successively with $\gamma=\gamma_i$, $0\leq i\leq j_\infty$ and $\e$ replaced by $\frac{\e(p-1)}{2}$ and proceeding similarly as in the proof of Theorem \ref{teo:localalmost} to obtain\footnote{Note that in the case $\Gamma=1$, we have $\gamma_{i_\infty-1}<1-\e$ so that Proposition \ref{prop:improve2III} applied with $\gamma=\gamma_{j_\infty-2}$ yields the exponent 
$$
\min(\gamma_{j_\infty-1},1-\e (p-1)/2)=\gamma_{j_\infty-1},
$$
since $\gamma_{j_\infty-1}<1-\e$. Therefore, the iteration scheme is intact up this step.} 
\[
\begin{split}
2^{\gamma_1}[v_1]_{C^{\gamma_1}(B_1)}=[v_0]_{C^{\gamma_1}(B_\frac12)}
\leq CM_0(1+M_0^{q-p}A)^{d_0}\leq C(s,p,t,q,\e,N)(1+A)^{d_0},\\
\end{split}
\]
\[
\begin{split}
2^{\gamma_{i+1}}[v_{i+1}]_{C^{\gamma_{i+1}}(B_1)}&=[v_i]_{C^{\gamma_{i+1}}(B_\frac12)}\leq M_i C(1+2^{-i(sp-tq)}M_i^{q-p} A)^{d_i}\\
&\leq  C(1+A^{d_i})(1+[v_i]_{C^{\gamma_i}(B_1)}^{(q-p){d_i}+1})\\
\end{split}
\]
and finally
\[
\begin{split}
 2^{\Gamma-\e}[v_{j_{\infty}}]_{C^{\Gamma-\e}(B_\frac12)}&\leq [v_{j_{\infty}-1}]_{C^{\min(\gamma_{j_{\infty}},1-\frac{\e(p-1)}{2})}(B_\frac12)}\leq M_{j_\infty-1} C(1+2^{(-j_\infty+1)(sp-tq)}M_{j_\infty-1}^{q-p} A)^{ d_{j_\infty-1} }\\
&\leq  C(1+ A^{d_{j_\infty-1}})(1+[v_{j_\infty-1}]_{C^{\gamma_{j_\infty-1}}(B_1)}^{(q-p){d_{j_\infty-1}}+1}).
\end{split}
\]
As in the proof of Theorem \ref{teo:localalmost} this yields after scaling back
$$
[u]_{C^{\Gamma-\e}(B_{2^{-j_\infty-1}})}=2^{j_\infty(\Gamma-\e)}[v_{j_{\infty}}]_{C^{\Gamma-\e}(B_\frac12)}\leq C(1+A^{\beta}),
$$
where $\beta= \beta(N,s,p,q,\e)\geq 2 $ and $C=C(N,s,p,t,q,\e)>0$. This is the desired result with $\sigma = 2^{-j_\infty-1}$.\\
\textbf{Case-II: When $1<q\leq p\leq 2$:}\\
Taking into account \eqref{newbound}, we apply Proposition \ref{prop:improve2III} to $w_i=v_i/M_i$ successively with $\gamma=\gamma_i$, $0\leq i\leq j_\infty$ and $\e$ replaced by $\frac{\e(p-1)}{2}$ and proceeding similarly as in the proof of Theorem \ref{teo:localalmost} to obtain
\[
\begin{split}
2^{\gamma_1}[v_1]_{C^{\gamma_1}(B_1)}=[v_0]_{C^{\gamma_1}(B_\frac12)}\leq CM_0(1+M_0^{q-p}A)^{d_0}\leq C(s,p,t,q,\e,N)(1+A)^{d_0},
\end{split}
\]
\[
\begin{split}
2^{\gamma_{i+1}}[v_{i+1}]_{C^{\gamma_{i+1}}(B_1)}&=[v_i]_{C^{\gamma_{i+1}}(B_\frac12)}\leq M_i C(1+2^{-i(sp-tq)}M_i^{q-p} A)^{d_i}\\&\leq C(1+A^{d_i})(1+[v_i]_{C^{\gamma_i}(B_1)}) 
\end{split}
\]
and finally
\[
\begin{split}
 2^{\Gamma-\e}[v_{j_{\infty}}]_{C^{\Gamma-\e}(B_\frac12)}&\leq [v_{j_{\infty}-1}]_{C^{\min(\gamma_{j_{\infty}},1-\frac{\e(p-1)}{2})}(B_\frac12)}\leq M_{j_\infty-1} C(1+2^{(-j_\infty+1)(sp-tq)}M_{j_\infty-1}^{q-p} A)^{d_{j_\infty-1}}\\
&\leq  C(1+ A^{d_{j_\infty-1}})(1+[v_{j_\infty-1}]_{C^{\gamma_{j_\infty-1}}(B_1)}).
\end{split}
\]
Here we have used that $q\leq p$ to get rid of the factors containing $M_i$. Iteration of the above inequalities implies, after scaling back
$$
[u]_{C^{\Gamma-\e}(B_{2^{-j_\infty-1}})}=2^{j_\infty(\Gamma-\e)}[v_{j_{\infty}}]_{C^{\Gamma-\e}(B_\frac12)}\leq C(1+A^{\beta}),
$$
where $\beta=\beta(N,s,p,q,\e)\geq 2$ and $C=C(N,s,p,t,q,\e)>0$.  Setting $\sigma = 2^{-j_\infty-1}$, this yields the desired estimate. 
\end{proof}

\subsection{Proof of the main result: Theorem \ref{teo:1higher}}
The proof of the main H\"older regularity result now easily follows. We spell out the details.
\begin{proof}
By \cite[Proposition 3.3]{GKS}, we have $u\in L^{\infty}_{\mathrm{loc}}(\Omega)$.  Without loss of generality, we may assume $x_0=0$.   Let
\[
w(x):=\frac{1}{\mathcal{M}_r}\,u(rx),\qquad \mbox{ for }x\in B_2,
\]
where
\[
\mathcal{M}_r=\|u\|_{L^\infty(B_{r})}+\mathrm{Tail}_{p-1,s\,p}(u;0,r)+\mathrm{Tail}_{q-1,t\,q}(u;0,r)+1>0.
\]
 It is straightforward to see that $w$ solves $(-\Delta_p)^s u+r^{sp-tq}\mathcal{M}_r^{q-p}(-\Delta_q)^t u=0$ in $B_2$ and enjoys the bounds
\begin{equation*}
\label{assumption}
\|w\|_{L^\infty(B_1)}\leq 1,\qquad \int_{\mathbb{R}^N\setminus B_1}\frac{|w(y)|^{l-1}}{|y|^{N+l\,m}}\,  dy\leq 1,
\end{equation*}
where $(l,m)\in\{(p,s),(q,t)\}$.  Theorem \ref{teo:localalmostIII} implies 
\[
[w]_{C^{\Gamma-\e}(B_{\sigma })}\leq C(1+(Ar^{sp-tq}\mathcal{M}_r^{q-p})^{\beta}),
\]
where $C=C(N,s,p,t,q,\e)>0,\, \beta=\beta(N,s,p,q,\e)\geq 1$ and $\sigma=\sigma(\e,s,p)>0$.  Rewritten in terms of $u$, this is the desired estimate. 
\end{proof}

\section{Acknowledgment}
P.G. thanks IISER Berhampur for the seed grant: IISERBPR/RD/OO/2024/15, Date: February 08, 2024. E.L. has been supported by the Swedish Research Council, grant no. 2023-03471.

\medskip
\bibliographystyle{plain}
\bibliography{bibfile}

\noindent {\textsf{Prashanta Garain\\Department of Mathematical Sciences\\
Indian Institute of Science Education and Research Berhampur\\ Permanent Campus, At/Po:-Laudigam,\\
Dist.-Ganjam, Odisha, India-760003, 
}\\ 
\textsf{e-mail}: pgarain92@gmail.com\\

\noindent {\textsf{Erik Lindgren\\  Department of Mathematics\\ KTH -- Royal Institute of Technology\\ 100 44, Stockholm, Sweden}  \\
\textsf{e-mail}: eriklin@math.kth.se\\

\end{document}